\documentclass[pdflatex,sn-mathphys,Numbered]{sn-jnl}

\usepackage{graphicx}%
\usepackage{nomencl}
\usepackage{multirow}%
\usepackage{amssymb,amsmath,ascmac,array,bm}%
\usepackage{amsthm}%
\usepackage{mathrsfs}%
\usepackage[title]{appendix}%
\usepackage{xcolor}%
\usepackage{textcomp}%
\usepackage{manyfoot}%
\usepackage{booktabs}%
\usepackage{algorithm}%
\usepackage{algorithmicx}%
\usepackage{algpseudocode}%
\usepackage{listings}%

\usepackage{algorithmicx,algorithm}
\usepackage{algpseudocode}
\usepackage{enumitem}
\usepackage{booktabs}  
\usepackage{threeparttable,multirow,array} 
\usepackage[most]{tcolorbox}
\usepackage{threeparttable,multirow,array} 

\usepackage{bbm}
\usepackage{soul}

\usepackage{mathtools}
\usepackage{fancyvrb}

\mathtoolsset{showonlyrefs} 

\usepackage{tikz}
\usetikzlibrary{positioning,shapes,shadows,arrows}

\newtheorem{thm}{Theorem}[section]

\newtheorem{lem}[thm]{Lemma}

\newtheorem{rem}[thm]{Remark}
\newtheorem{assump}[thm]{Assumption}

\numberwithin{equation}{section}
\allowdisplaybreaks

\newcommand{\mca}{\mathcal{A}}
\newcommand{\mcb}{\mathcal{B}}
\newcommand{\mcc}{\mathcal{C}}

\newcommand{\mcf}{\mathcal{F}}

\newcommand{\mcl}{\mathcal{L}}

\newcommand{\mbbh}{\mathbb{H}}
\newcommand{\mbbn}{\mathbb{N}}

\newcommand{\mbbr}{\mathbb{R}}

\newcommand{\mbby}{\mathbb{Y}}

\newcommand{\mbbzp}{\mathbb{Z}_{+}}

\newcommand{\al}{\alpha}  \newcommand{\ep}{\epsilon} 
\newcommand{\ve}{\varepsilon} 
\newcommand{\vp}{\varphi} \newcommand{\del}{\delta} \newcommand{\sig}{\sigma}
\newcommand{\D}{\Delta}  
\newcommand{\Lam}{\Lambda} \newcommand{\Gam}{\Gamma}

\newcommand{\p}{\partial}
\newcommand{\ra}{\rangle} \newcommand{\la}{\langle}
\newcommand{\cil}{\xrightarrow{\mcl}} 
\newcommand{\cip}{\xrightarrow{p}} 

\newcommand{\argmax}{\mathop{\rm argmax}}

\def\ds#1{\displaystyle{#1}}

\def\nn{\nonumber}

\def\hmrev#1{\textcolor{black}{#1}}	
\def\hmrrev#1{\textcolor{black}{#1}}	
\def\yurev#1{\textcolor{black}{#1}}	

\newcommand{\pr}{P} \newcommand{\E}{E}

\def\var{{\rm var}}

\def\diag{{\rm diag}}

\def\lp{L\'evy process}
\def\lm{L\'evy measure}
\def\ld{L\'evy density}

\def\wp{Wiener process}
\def\cadlag{c\`adl\`ag}

\newcommand{\sumjj}{\sum_{j=1}^{N_n}}

\newcommand{\tz}{\theta_{0}}

\newcommand{\nes}{\hat{\nu}_{n}}

\newcommand{\mes}{\hat{\mu}_{n}}
\newcommand{\ses}{\hat{\sig}_{n}}
\newcommand{\hep}{\hat{\varepsilon}}
\newcommand{\tep}{\tilde{\varepsilon}}

\newcommand{\sumi}{\sum_{i=1}^{[T_n]}}

\newcommand{\yuima}{YUIMA }

\DefineVerbatimEnvironment{Sinput}{Verbatim}{fontshape=sl}
\DefineVerbatimEnvironment{Soutput}{Verbatim}{}
\DefineVerbatimEnvironment{Scode}{Verbatim}{fontshape=sl}

\DefineVerbatimEnvironment{Code}{Verbatim}{}
\DefineVerbatimEnvironment{CodeInput}{Verbatim}{fontshape=sl}
\DefineVerbatimEnvironment{CodeOutput}{Verbatim}{}

\setkeys{Gin}{width=0.8\textwidth}

\begin{document}

\title[Quasi-likelihood analysis for Student-L\'{e}vy regression]
{Quasi-likelihood analysis for Student-L\'{e}vy regression}


\author*[1]{\fnm{Hiroki} \sur{Masuda}}\email{hmasuda@ms.u-tokyo.ac.jp}

\author[2]{\fnm{Lorenzo} \sur{Mercuri}}\email{lorenzo.mercuri@unimi.it}
\equalcont{These authors contributed equally to this work.}

\author[3]{\fnm{Yuma} \sur{Uehara}}\email{y-uehara@kansai-u.ac.jp}
\equalcont{These authors contributed equally to this work.}

\affil*[1]{\orgdiv{Graduate School of Mathematical Sciences}, \orgname{University of Tokyo}, \orgaddress{\country{Japan}}}

\affil[2]{\orgdiv{Department of Economics, Management and Quantitative Methods}, \orgname{University of Milan}, \orgaddress{\country{Italy}}}

\affil[3]{\orgdiv{Department of Mathematics, Faculty of Engineering Science}, \orgname{Kansai University}, \orgaddress{\country{Japan}}}


\abstract{
We consider the quasi-likelihood analysis for a linear regression model driven by a Student-$t$ L\'{e}vy process with constant scale and arbitrary degrees of freedom. The model is observed at high frequency over an extending period, under which we can quantify how the sampling frequency affects estimation accuracy. In that setting, joint estimation of trend, scale, and degrees of freedom is a non-trivial problem. The bottleneck is that the Student-$t$ distribution is not closed under convolution, making it difficult to estimate all the parameters fully based on the high-frequency time scale. To efficiently deal with the intricate nature from both theoretical and computational points of view, we propose a two-step quasi-likelihood analysis: first, we make use of the Cauchy quasi-likelihood for estimating the regression-coefficient vector and the scale parameter; then, we construct the sequence of the unit-period cumulative residuals to estimate the remaining degrees of freedom.
In particular, using full data in the first step causes a problem stemming from the small-time Cauchy approximation, showing the need for data thinning.
%
}




\maketitle

\section{Introduction}
\label{Intro}

Suppose that we have a discrete-time sample $\{(X_{t_j},Y_{t_j})\}_{j=0}^{[nT_n]}$ with $t_j=j/n$ from the continuous-time (location) regression model
\begin{equation}
Y_t = X_t \cdot \mu + \sig J_t, \qquad t\in[0,T_n],
\label{hm:model}
\end{equation}
where $X=(X_t)$ is a {\cadlag} stochastic covariate process in $\mbbr^q$ satisfying some regularity conditions, the dot denotes the inner product in $\mbbr^q$, and $J=(J_t)$ is a {\lp} such that the unit-time distribution 
\begin{equation}
\mcl(J_1) = t_\nu:=t_\nu(0,1),
\nonumber
\end{equation}
where $t_\nu(\mu,\sig)$ denotes the scaled Student-$t$ distribution with the density
\begin{equation}
f(x;\mu,\sig,\nu) := \frac{\Gamma(\frac{\nu+1}{2})}{\sig\sqrt{\pi}\Gamma(\frac{\nu}{2})}
\left\{1+\left(\frac{x-\mu}{\sig}\right)^2\right\}^{-(\nu+1)/2}.
\label{hm:def_f.nu.pdf}
\end{equation}
Our objective is to estimate the true value $\tz=(\mu_0,\sig_0,\nu_0)\in\Theta$ from $\{(X_{t_j},Y_{t_j})\}_{j=0}^{[nT_n]}$ when $T_n\to\infty$ for $n\to\infty$.
We will write $h=h_n=1/n$ for the sampling step size.
All the processes are defined on an underlying filtered probability space $(\Omega,\mcf,(\mcf_t)_{t\ge 0},\pr)$.
The corresponding statistical model is indexed by the unknown parameter $\theta := (\mu,\sig,\nu) \in \Theta_\mu\times\Theta_\sig \times \Theta_\nu=\Theta$. Throughout, we assume that $\Theta$ is a bounded convex domain in $\mbbr^{q+2}$ such that its closure $\overline{\Theta} \subset \mbbr^q \times (0,\infty) \times (0,\infty)$.

Since the Student-$t$ distribution is not closed under convolution, $\mcl(J_h)$ for $h\ne 1$ is no longer Student-$t$-distributed.
The exact likelihood function is given only through the Fourier inversion, resulting in the rather intractable expression:
the characteristic function of $\mcl(J_1)=t_\nu$ is given by
\begin{equation}
\vp_{J_1,\nu}(u) := \frac{2^{1-\nu/2}}{\Gam(\nu/2)} |u|^{\nu/2} K_{\nu/2}(|u|), \qquad u\in\mbbr,
\label{hm:t_nu(0,1)-CF}
\end{equation}
hence $\mcl(J_h)$ admits the Lebesgue density
\begin{align}
x &\mapsto 
\frac{1}{\pi}\int_0^\infty \cos(ux)\{\vp_{J_1,\nu}(u)\}^h du
\nn\\
&= \frac{1}{\pi}\int_0^\infty \cos(ux)\left(
\frac{2^{1-\nu/2}}{\Gam(\nu/2)} u^{\nu/2} K_{\nu/2}(u)
\right)^h du
\nn\\
&= \left(\frac{2^{1-\nu/2}}{\Gam(\nu/2)} \right)^h
\frac{1}{\pi}\int_0^\infty \cos(ux) u^{\nu h/2} \left(K_{\nu/2}(u)\right)^h du,
\label{hm:ff_def_pre}
\end{align}
where $K_\nu(t)$ denotes the modified Bessel function of the second kind ($\nu\in\mbbr$, $t>0$):
\begin{equation}
K_{\nu}(t)=\frac{1}{2}\int_0^{\infty}s^{\nu-1}\exp\left\{-\frac{t}{2}\left(s+\frac{1}{s}\right)\right\}ds.
\nonumber
\end{equation}
Numerical integration in \eqref{hm:ff_def_pre} can be unstable for very small $h$ and also for large $|x|$.

Previously, the thesis \cite{Massing2019} studied in detail the local asymptotic behavior of the log-likelihood function with numerics for a sample-path generation about Student-$t$ {\lp} observed at high frequency. However, the degrees-of-freedom parameter $\nu$ was supposed to be known and to be greater than $1$ throughout.

In this paper, we propose the explicit two-stage estimation procedure:
\begin{enumerate}
\item First, in Section \ref{hm:sec_Cauchy.QLA}, we construct an estimator $(\mes,\ses)$ of $(\mu,\sig)$ based on the \textit{Cauchy quasi-likelihood};
\item Then, in Section \ref{hm:sec_student.LA}, we make use of the \textit{Student-$t$ quasi-likelihood} for construction of an estimator $\nes$ of $\nu$, 
through the ``unit-time'' residual sequence
\begin{equation}
\hep_i := \ses^{-1}\left(Y_i - Y_{i-1} - \mes\cdot(X_i - X_{i-1})\right),
\nonumber
\end{equation}
for $i=1,\dots,[T_n]$, which is expected to be approximately i.i.d. $t_\nu$-distributed.
\end{enumerate}
\hmrev{
\yurev{For both estimators}, we will prove the asymptotic normality and convergence of their moments (Theorems \ref{hm:thm_CQLA}, \ref{hm:thm_tQLA}, and \ref{hm:thm_joint.AN}).
Also discussed is how we can weaken the regularity conditions when only interested in deriving the asymptotic normality (Theorems \ref{hm:thm_CQMLE.AN} and \ref{hm:thm_CQMLE.AN-SDE}).
}
The important point is that we can asymptotically efficiently estimate $(\mu,\sig)$ with effectively leaving $\nu$ unknown;
under additional conditions, the least-squares estimator is asymptotically normal, but the associated efficiency loss can be significant (see Remark \ref{hm:rem_LSE} below).
We refer to \cite{IvaKulMas15} for the details of the corresponding LAN property.
In our study, the rate of convergence of $(\mes,\ses)$ in the first stage depends on how much data we use through the sequence $(B_n)$ introduced later in \eqref{hm:B_order}.
Under suitable conditions, the foregoing two-stage estimation procedure enables us to estimate $\nu_0$ as if we directly observe the latent unit-time noise sequence $(J_{i}-J_{i-1})_{i=1}^{[T_n]}$.
Since the characteristic function and the {\lm} of the Student-$t$ distribution are both intractable, we believe that 
using the Student-$t$ likelihood function based on the unit-time residuals is natural from the computational viewpoint.
We will present some numerical experiments in Section \ref{sec_numerics}.
The proofs are gathered in Section \ref{hm:sec_proofs}.

Note that $X$ may be random and that in some situations $X$ and $J$ may be stochastically dependent on each other. 
Our stepwise procedure with different time scales makes the optimization much easier than the joint estimation.
The proposed estimators $\ses$ and $\nes$ are asymptotically independent (Theorem \ref{hm:thm_joint.AN}) while the maximum-likelihood estimators in the conventional i.i.d. and time-series settings are asymptotically correlated; see \cite[Section 2.2]{Har13} for details.

\hmrev{
Below, we briefly list the necessary restrictions in our approach without details.
\begin{itemize}
\item First, we need $T_n\to\infty$ in the second step for estimating $\nu_0$.
An asymptotically efficient estimation of $\nu$ based on the full high-frequency sample $\{(X_{t_j},Y_{t_j})\}_{j=0}^{[nT_n]}$ is a non-trivial problem, and we do not know even whether or not the derived rate $\sqrt{T_n}$ is optimal for estimating $\nu_0$. We do not address it in this paper.
\item Second, the terminal sampling time $T_n$ should not grow so quickly if we use whole data in $[0,T_n]$; otherwise, we need thinning data through the sequence $(B_n)$ satisfying the sampling-balance conditions \eqref{hm:B_order} and \eqref{T/N->0} below.
This implies that the accumulated Cauchy-approximation error (see Lemma \ref{hm:lem_loc.lim.thm}) for rapidly growing $T_n$ may 
crush the clean-cut asymptotic behavior of the estimator $\nes$ described in Theorem \ref{hm:thm_tQLA} (hence also those in Theorems \ref{hm:thm_joint.AN}, \ref{hm:thm_CQMLE.AN}, and \ref{hm:thm_CQMLE.AN-SDE}).
We also imposed the additional sampling-design condition \eqref{B/T->0} to establish the asymptotic orthogonality between the proposed estimators of $(\mu_0,\sig_0)$ and $\nu_0$.
\end{itemize}
}

\medskip

We will denote by $\pr_\theta$ the underlying probability measure for $(J, X, Y)$ associated with $\theta$, and by $\E_\theta$ the expectation with respect to $\pr_\theta$. Unless otherwise mentioned, any asymptotics will be taken under $\pr:=\pr_{\tz}$ for $n\to\infty$.
We will denote by $C$ a universal positive constant which is independent of $n$ and may vary whenever it appears.
The partial differentiation with respect to the variable $x$ will be denoted by $\p_x$; we simply use \hmrrev{$\p$} without specifying a variable if not confusing.
The notation $a_n\lesssim b_n$ for positive sequences $(a_n)$ and $(b_n)$ means that $\limsup_n(a_n/b_n)<\infty$.
We denote by $I_A=I(A)$ the indicator function of a set $A$.

\section{Cauchy quasi-likelihood analysis}
\label{hm:sec_Cauchy.QLA}

\subsection{Construction and asymptotic results}

By the expression \eqref{hm:t_nu(0,1)-CF} and the asymptotic property $K_\nu(z)\sim \sqrt{\pi/(2z)} e^{-z}$ for $z\to \infty$, 
the characteristic function of $h^{-1}J_h$ equals for $h\to 0$,
\begin{align}
\left(\E_\theta\left[\exp(iu h^{-1}J_1)\right]\right)^h
&= (1+o(1)) \big|u\nu^{-1/2}h^{-1}\big|^{\nu h/2} K_{\nu/2}\big(|uh^{-1}|\big)^h
\nn\\
&= (1+o(1)) \left(
\sqrt{\frac{\pi h}{2|u|}} \exp\left( - h^{-1}|u|\right)(1 + o(1))\right)^h
\nn\\
&\to \exp(-|u|).
\label{hm:Cauchy.approx_CF}
\end{align}
This implies the weak convergence $h^{-1}J_h \cil t_1$ (the standard Cauchy distribution) as $h\to 0$, whatever the degrees of freedom $\nu_0>0$ is.
It is therefore natural to construct the Cauchy quasi-(log-)likelihood which is defined as if the conditional distribution $Y_{t_j} | \{Y_{t_{j-1}}=y,\, X_{t_{j-1}},\, X_{t_j}\}$ under $\pr_\theta$ equals the Cauchy distribution with location $y+ \mu\cdot(X_{t_j}-X_{t_{j-1}})$ and scale $h\sig$; note that the conditional likelihood is misspecified unless $\nu_0=1$.
There are both advantages and disadvantages:
on the one hand, we can make an inference for $(\mu,\sig)$ without knowing the value $\nu_0$, but on the other hand, the information of $\nu_0$ disappeared in the small-time limit.

Using the Cauchy quasi-likelihood is not free:
we have to manage the Cauchy-approximation error for $\mcl(h^{-1}J_h)$ in the mode of $L^1(dy)$-local-limit theorem.
Since the approximation error accumulates for $T_n\to\infty$, using the whole sample $\{(X_{t_j},Y_{t_j})\}_{j=0}^{[n T_n]}$ may introduce too much error disrupting the Cauchy approximation; this will be seen from the proof of Theorem \ref{hm:thm_CQLA} below. 
Depending on the situation, we will need either to thin the sample or to control the speed of $T_n\to\infty$.

We consider the (possibly) partial observations over the part $[0,B_n]$ of the entire period $[0, T_n]$, where $(B_n)$ is a positive sequence such that
\begin{equation}
B_n\le T_n, 
\qquad n^{\ep''} \lesssim B_n \lesssim n^{1-\ep'}
\label{hm:B_order}
\end{equation}
for some $\ep',\ep''\in(0,1)$.
We write the corresponding number of time points used as
\begin{equation}
N_n = [n B_n].
\label{hm:N_def}
\end{equation}
Lemma \ref{hm:lem_loc.lim.thm} below implies that we cannot always take the whole sample (namely $B_n=T_n$) to manage the error of the local-Cauchy approximation.
In other words, when one uses the whole sample for estimating $(\mu,\sig)$, the condition \eqref{hm:B_order} restricts the speed of $T_n\to\infty$ as $T_n=O(n^{1-\ep'})$.

Write $a=(\mu,\sig)$, $a_0=(\mu_0,\sig_0)$, and $\D_j\xi = \xi_{t_j} - \xi_{t_{j-1}}$ for any process $\xi$.
Let
\begin{equation}
\ep_j(a) := \frac{\D_j Y - \mu\cdot \D_j X}{h\sig}.
\nonumber
\end{equation}
\hmrev{
Let $\phi_1(y):=\pi^{-1}(1+y^2)^{-2}$, the standard Cauchy density.
}
Then, we introduce the Cauchy quasi-(log-)likelihood conditional on $X$:
\begin{align}
\mbbh_{1,n}(a) 
&:= \sum_{j=1}^{N_n} \log\left\{\frac{1}{h\sig}\phi_1\left( \frac{\D_j Y - \mu\cdot \D_j X}{h\sig} \right)\right\}
\nn\\
&= C_n - \sum_{j=1}^{N_n} \left\{ \log\sig + \log\left(1+\ep_j(a)^2\right)\right\},
\nonumber
\end{align}
where the term $C_n$ does not depend on $a$.
We define the Cauchy quasi-maximum likelihood estimator (CQMLE for short) by any element
\begin{equation}
\hat{a}_n:=(\mes,\ses)\in\argmax_{a\in \overline{\Theta_\mu\times\Theta_\sig}}\mbbh_{1,n}(a).
\label{hm:def_CQMLE}
\end{equation}

We introduce the following regularity conditions on $X$.

\begin{assump}
\label{hm:A_X}
For $X$, we can associate an $(\mcf_t)$-adapted process $t\mapsto X'_t$ in $\mbbr^q$ having {\cadlag} 
sample paths, for which the following conditions hold for every $K>0$:
\begin{enumerate}
\item $\ds{\sup_{t\ge 0}\E\big[|X'_t|^K\big]<\infty}$.
\item There exists a constant $c_X>0$ such that both
\begin{equation}
\sup_{n} \max_{j=1,\dots, [n T_n]} \E\left[\left|\frac{1}{h^{c_X}}\left\{\frac1h \left(\D_j X  - h X'_{t_{j-1}}\right)\right\} \right|^K\right] < \infty
\label{hm:A_X-1}
\end{equation}
and
\begin{equation}
\sup_{n} \max_{j=1,\dots, [n T_n]} \sup_{t_{j-1}\le t\le t_j}\E\left[\left|\frac{1}{h^{c_X}} (X_t' - X'_{t_{j-1}}) \right|^K\right] < \infty
\label{hm:A_X-2}
\end{equation}
hold.
Moreover,
\begin{equation}
\max_{j=1,\dots, [n T_n]} \E\left[\left|
\sqrt{N_n}\E\left.\left[\left\{\frac1h \left(\D_j X  - h X'_{t_{j-1}}\right)\right\}
\,\frac{\p\phi_1}{\phi_1}\left(\frac{\D_j J}{h}\right)\right|\mcf_{t_{j-1}}\right]
\right|^K\right] =o(1).
\nonumber
\end{equation}
%
%

\item 
\hmrev{
There exist a constant $c_X'>0$ and a probability measure $\pi_0(dx)$ on some $q'$-dimensional Borel space $(\mbbr^{q'},\mcb^{q'})$ such that, for every real-valued $\mcc^1(\mbbr^q)$-function $f(x)$ satisfying $\max_{k\in\{0,1\}}|\p_x^k f(x)| \lesssim 1 + |x|^C$, we can associate a measurable function $\psi_f(z)$ such that
\begin{equation}
\sup_{n} \E\left[\left| N_n^{c_X'}\left(\frac{1}{N_n} 
\sum_{j=1}^{N_n} f(X'_{t_{j-1}}) - \int \psi_f (z)\pi_0(dz)\right)\right|^K \right] < \infty
\label{hm:A_X-3}
\end{equation}
for every $K>0$.
}

\item 
\hmrev{
The $q\times q$ matrix $\ds{S_0 := \int \psi_f (z)\pi_0(dz)}$ for $f(x')=x^{\prime \,\otimes 2}$ is symmetric, positive-definite, and finite.
}

\end{enumerate}
\end{assump}

\medskip

\hmrev{
Assumption \ref{hm:A_X} is designed not only for convergence in distribution of our estimators but also for ensuring the convergence of their moments: 
in the proofs, items 1 and 2 are used for a series of moment estimates, item 3 (together with the Sobolev inequality) is required to handle the uniform-in-parameter moment estimates and to identify the limit of the law of large numbers under the 
ergodicity, and item 4 will ensure the positive definiteness of the asymptotic covariance matrix of the estimator.
It should be noted that the moment conditions in Assumption \ref{hm:A_X} can be too much to ask if one is only interested in deriving the asymptotic normality; see Theorems \ref{hm:thm_CQMLE.AN} and \ref{hm:thm_CQMLE.AN-SDE}.
}

\hmrev{
The last convergence in Assumption \ref{hm:A_X}.2 is trivial if $X'$ and $J$ are independent so that the conditional expectation involved therein vanishes a.s.
Although we did not explicitly mention, one may think of $X_t=\int_0^t X'_s ds$ and where the integrand $X'$ admits a unique invariant distribution $\pi_0$ for which $X_t' \cil \pi_0$ as $t\to\infty$ and $S_0 = \int f(x',a)\pi_0(dx')$ (in this case, $q=q'$ and $\psi_f = f$);} 
then, since $\int f(x')\pi_0(dx') \le \sup_{t\ge 0}\E[f(X'_t)]$ for any continuous nonnegative $f$, we have $\int |x'|^K\pi_0(dx')<\infty$ for each $K>0$ hence in particular $S_0$ is finite.
\hmrev{
See Section \ref{hm:sec_cov.proc} for more related discussions; 
as we will see there, the dimension $q'$ in Assumption \ref{hm:A_X}.3 may differ from $q$.
}

Let
\begin{align}
\D_{a,n} &:= \frac{1}{\sqrt{N_n}} \p_a\mbbh_{1,n}(a_0), \nn\\
\Gam_{a,0} &:= \diag\left(
\frac{1}{2\sig_0^2} \,S_0, ~\frac{1}{2\sig_0^2}
\right).
\label{hm:def_Gam_a}
\end{align}
\hmrrev{
Let
\begin{equation}
\zeta(y,x';\,a)
:= \log\left\{\frac{\sig_0}{\sig}\phi_1\left( \frac{\sig_0}{\sig}\left(y - \frac{1}{\sig_0}(\mu-\mu_0)\cdot x'\right) \right)\right\}
- \log\phi_1(y),
\label{hm:zeta_def}
\end{equation}
and
\begin{equation}
f'(x';a) := \int \zeta(y, x';a) \phi_1(y)dy,
\nonumber
\end{equation}
which is smooth in $a\in\overline{\Theta_\mu\times\Theta_\sig}$.
Then, under Assumption \ref{hm:A_X}.3, for each $a$ we can associate measurable functions 
$\psi_{f'}(z;a)$ and $\psi_{\p_a f'}(z;a)$ and a probability measure $\pi_0(dz)$, for which
\begin{equation}
\frac{1}{N_n}\sum_{j=1}^{N_n} \p_a^k f'(X'_{t_{j-1}};a) \cip \int \psi_{\p_a^k f'(\cdot;a)} (z;a)\pi_0(dz),\qquad k=0,1.
\nn
\end{equation}
We write
\begin{equation}
\mbby_1(a)
= \int \psi_{f'(\cdot;a)} (z;a)\pi_0(dz).
\label{hm:def_new.Y1}
\end{equation}
}
The main result of this section is the following.

\begin{thm}
\label{hm:thm_CQLA}
\hmrrev{
Suppose that \eqref{hm:B_order} and Assumption \ref{hm:A_X} hold, that $a\mapsto \mbby_1(a)$ is $\mcc^1$-class with $\p_a\mbby_1(a)=\int \psi_{\p_a f'(\cdot;a)} (z;a)\pi_0(dz)$, and that there exists a constant $c_0>0$ for which
\begin{equation}
\label{hm:QLA.thm1-4}
\mbby_{1}(a) \le -c_{0} |a-a_0|^{2},\qquad a\in\Theta_\mu\times\Theta_\sig.
\end{equation}
} 
Then, we have the asymptotic normality 
\begin{equation}
\hat{u}_{a,n} := \sqrt{N_n}(\hat{a}_n-a_0) = \Gam_{a,0}^{-1}\D_{a,n} + o_p(1) \cil N\left(0,\,\Gam_{a,0}^{-1}\right)
\label{hm:CQMLE.an}
\end{equation}
and the polynomial-type tail-probability estimate:
\begin{equation}
\forall L>0,\quad 
\sup_{r>0}\sup_n \pr\left[|\hat{u}_{a,n}| > r \right] r^L < \infty.
\label{hm:CQMLE.tpe}
\end{equation}
\end{thm}

\hmrrev{
If $q'=q$, $\pi_0$ is the invariant distribution of $X'$, and if Assumption \ref{hm:A_X}.3 and \ref{hm:A_X}.4 hold with $\psi_f=f$, then 
\begin{equation}
\mbby_1(a) = \int \int \zeta(y, x';a) \phi_1(y)dy \,\pi_0(dx').
\nonumber
\end{equation}
In this case, the identifiability condition \eqref{hm:QLA.thm1-4} is automatic.
Indeed, the well-known property of the Kullback-Leibler divergence, we have $\mbby_1(a)\le 0$ with the equality holding if and only if $a=a_0$. Moreover, the above $\mbby_1(a)$ is smooth, $\p_a\mbby_1(a_0)=0$, and $-\p_a^2\mbby_1(a_0)$ is positive definite under the present assumptions (including the positive definiteness of $S_0 = \int x^{\prime \otimes 2}\pi_0(dx')$). These observations conclude \eqref{hm:QLA.thm1-4}.
}

It is easy to construct a consistent estimator of $S_0$: see \eqref{hm:S-hat} in Section \ref{hm:sec_jan}.
We have the variance-stabilizing transform for $\sig$: $\sqrt{N_n /2}(\log\ses - \log\sig_0) \cil N_1(0,1)$.

\medskip

\begin{rem}\normalfont
Making use of sample in over $[0,B_n] \subset [0,T_n]$ is not essential.
It is possible to use, for example, only sample over $\bigcup_{k\ge 1}[2k-2,2k-1)$.
Also, we could consider the case where $B_n\equiv B$ for some \textit{fixed} $B>0$, meaning that for estimating $(\mu,\sig)$ we only use a sample over the fixed period $[0, B]$. In that case, if $X$ is truly random, under suitable conditions on the underlying filtration $(\mcf_t)$, the asymptotic distribution of $\sqrt{n}(\hat{a}_n -a_0)$ is mixed-normal with the random asymptotic covariance depending on a sample path $(X_t)_{t\in[0, B]}$.
See also Remark \ref{hm:rem_AMN}, \cite{CleGlo19}, \cite{CleGlo20}, and \cite{Mas19spa} for related results.
\end{rem}

\begin{rem}\normalfont
Since high-frequency data over each fixed period is enough to consistently estimate $(\mu,\sig)$, 
one may think of a parameter-varying (randomized parameter) setting:
if $\mu$ and $\sig$ may vary along $i$, say $i\mapsto (\mu_i,\sig_i)$ for the $i$th period $[i-1,i)$ and if they are all to be estimated from $\{(X_{t_j},Y_{t_j})\}_{j\ge 1:\, jh\in [i-1,i)}$, then the model setup invokes the classical Neyman-Scott problem \cite{NeySco48};
if $\{(\mu_i,\sig_i)\}_{i\ge 1}$ is an unobserved i.i.d. sequence of random vectors, then the model becomes a (partially) random-effect one.
The latter would be of interest in the context of the population approach.
\end{rem}

\begin{rem}[Least-squares estimator]\normalfont
\label{hm:rem_LSE}
Let $\nu_0>2$. We can rewrite \eqref{hm:model} as
\begin{equation}
Y_t = X_t \cdot \mu + \sig' Z_t
\nonumber
\end{equation}
with $\sig':=\sig/\sqrt{\nu-2}$ and the {\lp} $Z_t:=\sqrt{\nu-2}\,J_t$ such that $\E[Z_1]=0$ and $\var[Z_1]=1$.
In this case, under suitable conditions, we can apply the least-squares method for estimating $(\mu,\sig')$ and deduce its asymptotic normality at rate $\sqrt{T_n}$.
\end{rem}

\begin{rem}\normalfont
The claims in Theorem \ref{hm:thm_CQLA} hold for any local-Cauchy $J$ satisfying the local limit theorem Lemma \ref{hm:lem_loc.lim.thm}.
In particular, the model \eqref{hm:model} with $J$ replaced by the generalized hyperbolic {\lp} (except for the variance gamma one) could be handled analogously; see \cite[Section 2.1.2]{Mas19spa} for related information.
\end{rem}

\subsection{Covariate process}
\label{hm:sec_cov.proc}

We set Assumption \ref{hm:A_X} without imposing a concrete structure of $X$.
To provide a set of more handy conditions on $X$ and $X'$, it is necessary to impose more specific structures on them.
\hmrrev{Below, we will present such an example in detail.}

\hmrrev{
\subsubsection{Setup}
}

Let us briefly discuss how to verify Assumption \ref{hm:A_X} for $X_t=(X_{1,t},X_{2,t})\in\mbbr^{q_1} \times \mbbr^{q_2}$ ($q=q_1+q_2$) of the form
\begin{equation}
t \mapsto X_{k,t} = \int_0^t X'_{k,s} ds,\qquad k=1,2,
\label{hm:X-ex-form}
\end{equation}
where $t\mapsto (X'_{1,t},X'_{2,t})$ is Riemann-integrable process independent of the Student-$t$ {\lp} $J$.
\hmrrev{Further, we assume the following conditions.}
\begin{itemize}
\item $X'_{1}$ is a non-random periodic $\mcc^1$-function with possibly unknown period $\tau>0$, and satisfies that
\begin{equation}
\sup_{t\ge 0}(|X_{1,t}|+|\p_t X_{1,t}|)<\infty.
\nonumber
\end{equation}
For example, $X'_1$ can be the derivative of 
\begin{equation}
X_{1,t}=\big(\cos(\tau_1 t),\sin(\tau_1 t),\dots,\cos(\tau_M t),\sin(\tau_M t)\big) 
\nn
\end{equation}
with $\tau_1,\dots,\tau_M$ being unknown different rational numbers ($q_1=2M$); we may get rid of some of the components from the beginning, such as $X_{1,t}=(\cos(5t),\sin(t))$.
\item $X'_{2}$ is a $q_2$-dimensional diffusion process:
\begin{equation}
dX'_{2,t} = A(X'_{2,t})dw_t + B(X'_{2,t})dt,
\label{hm:ex_X2}
\end{equation}
where $w$ is a standard {\wp} (independent of $J$).
The unknown coefficients $A$ and $B$ are globally Lipschitz so that \eqref{hm:ex_X2} admits a unique strong solution.
Moreover, $X'_2$ is 
exponentially ergodic in the sense that there exist some $\kappa>0$ and an invariant measure $\pi_{0}$ on $(\mbbr^{q_2},\mcb^{q_2})$ for which
\begin{align}
& \|P_t(x,\cdot) - \pi_0(dz)\|_g
\nn\\
&:=\sup_{t\ge 0} \sup_{f:\,|f|\le g} \left| \int f(z) P_t(x,dz) - \int f(z) \pi_{0}(dz)\right| \lesssim g(x)e^{-\kappa t}
\label{hm:rev1-2}
\end{align}
for any $g$ of at most polynomial growth.
\end{itemize}
Our model may be used for analyzing time series data observed at high frequency exhibiting a seasonality \cite{ProPed22}.
Concerned with the exponential ergodicity \eqref{hm:rev1-2} of $X'_2$, we refer to \cite{Kul18} and the references therein for several easy-to-verify sufficient conditions.

With the above setup, we will verify Assumption \ref{hm:A_X} along with imposing additional conditions.

\hmrrev{
\subsubsection{Verification}
}

First, for Assumption \ref{hm:A_X}.1 and \ref{hm:A_X}.2, we can look at $X_1$ and $X_2$ separately. 
The conditions obviously hold for $X_1$.
As for $X_2$ of \eqref{hm:ex_X2}, 
Assumption \ref{hm:A_X}.1 can be verified through the easy-to-apply Lyapunov-function criteria; we do not list them here, but just refer to \cite[Theorem 2.2]{Mas07}, \cite[Lemma 3.3]{Kul09}, and also \cite[Section 5]{Mas13as}.
Let us suppose Assumption \ref{hm:A_X}.1 in what follows.
Then, we can verify all the conditions in Assumption \ref{hm:A_X}.2 with $c_X=1/2$ as follows.
Since the sequence ($j=1,\dots,n$)
\begin{equation}
\frac{1}{\sqrt{h}} (X'_{2,t} - X'_{2,t_{j-1}})
= \frac{1}{\sqrt{h}}\int_{t_{j-1}}^{t}A(X'_{2,s})dw_s + \sqrt{h} \,\frac{1}{h}\int_{t_{j-1}}^{t}B(X'_{2,s})ds
\nonumber
\end{equation}
is $L^K(\pr)$-bounded for any $K>0$, so is 
\begin{equation}
\frac{1}{h\sqrt{h}} \left(\D_j X_2  - h X'_{2,t_{j-1}}\right)
= \frac1h \int_{t_{j-1}}^{t_j} \frac{1}{\sqrt{h}}(X'_{2,s} - X'_{2,t_{j-1}})ds,
\label{hm:add-eq1}
\end{equation}
followed by the first two conditions in Assumption \ref{hm:A_X}.2 with $c_X=1/2$.
\hmrev{
As for the last one, 
we note the decomposition
\begin{align}
& \frac{1}{h} \left(\D_j X_2  - h X'_{2,t_{j-1}}\right)
\nn\\
&= A(X'_{2,t_{j-1}}) \frac1h \int_{t_{j-1}}^{t_j}(w_s-w_{t_{j-1}})ds 
+ \bigg\{\frac1h \int_{t_{j-1}}^{t_j}\int_{t_{j-1}}^s B(X'_{2,u})du\,ds \nn\\
&{}\qquad 
+ \frac1h \int_{t_{j-1}}^{t_j}\int_{t_{j-1}}^s \big( A(X'_{2,u})-A(X'_{2,t_{j-1}})\big) dw_u\,ds\bigg\}\\
&\yurev{=:T_{1,h}+T_{2,h}.}
\label{hm:add-eq2}
\end{align}
Since $\sup_y |((\p\phi_1)/\phi_1)(y)| < \infty$ and 
$\E[((\p\phi_1)/\phi_1)(h^{-1}\D_j J)|\mcf_{t_{j-1}}]=0$ a.s., we have
\begin{align}
& \left|
\sqrt{N_n}\E\left.\left[\left\{\frac1h \left(\D_j X  - h X'_{t_{j-1}}\right)\right\}
\,\frac{\p\phi_1}{\phi_1}\left(\frac{\D_j J}{h}\right)\right|\mcf_{t_{j-1}}\right]
\right|
\nn\\
&\lesssim 
h\sqrt{N_n}\,\E\left.\left[\frac1h \big|\yurev{T_{2,h}}\big|\right|\mcf_{t_{j-1}}\right]
\nn\\
&{}\qquad + 
\left|
\sqrt{N_n}
A(X'_{2,t_{j-1}}) 
\E\left.\left[ \frac1h \int_{t_{j-1}}^{t_j}(w_s-w_{t_{j-1}})ds
\,\frac{\p\phi_1}{\phi_1}\left(\frac{\D_j J}{h}\right)\right|\mcf_{t_{j-1}}\right]
\right|.
\nonumber
\end{align}
The second term in the upper bound a.s. vanishes since $w$ and $J$ are mutually independent. 
Under the linear-growth property of the coefficients, by the standard estimates as before 
the sequence ``$\E[h^{-1}|\yurev{T_{2,h}}|\,|\,\mcf_{t_{j-1}}]$'' is $L^K(\pr)$-bounded for any $K>0$.
Since $h\sqrt{N_n}\lesssim n^{-\ep'/2}\to 0$ by \eqref{hm:B_order}, the last convergence in Assumption \ref{hm:A_X}.2 follows.
}
\hmrrev{Thus, we have verified Assumption \ref{hm:A_X}.1 and \ref{hm:A_X}.2.}

\hmrev{
Turning to Assumption \ref{hm:A_X}.3, we take $c'_X>0$ so small that $N_n^{\yurev{ c_X'}}\lesssim h^{-c_X}=h^{-1/2}$.
To handle the periodic nature of $X'_1$, we note that \eqref{hm:A_X-3} is derived by showing that
\begin{equation}
\E\left[\left|N_n^{c'_X}\left(\frac{1}{hN_n} \int_0^{hN_n} f(X'_t) dt - 
\int \psi_f (z)\pi_0(dz)\right)\right|^K \right] \lesssim 1
\label{hm:rev1-1}
\end{equation}
for suitable $\psi_f (z)$, since Jensen's and Cauchy-Schwarz inequalities give
\begin{align}
& \E\left[ \left|N_n^{c'_X}
\left(\frac{1}{hN_n} \int_0^{hN_n} f(X'_t) dt - \frac{1}{N_n} \sum_{j=1}^{N_n} f(X'_{t_{j-1}})\right)
\right| ^K\right]
\nn\\
&\le \E\left[\left|N_n^{c'_X}
\left(
\frac{1}{N_n} \sum_{j=1}^{N_n} \frac1h \int_{(j-1)/n}^{j/n}\big(f(X'_t) - f(X'_{t_{j-1}})\big) dt
\right)
\right| ^K\right]
\nn\\
&\lesssim 
\frac{1}{N_n} \sum_{j=1}^{N_n} \frac1h \int_{(j-1)/n}^{j/n}
\E\left[(1+|X'_t|^C+|X'_{t_{j-1}}|^C) \big(N_n^{c'_X}|X'_t - X'_{t_{j-1}}|\big)^K \right]dt
\nn\\
&\lesssim \frac{1}{N_n}\sumjj \frac1h \int_{t_{j-1}}^{t_j} \sup_t \E[ 1 + |X'_t|^C]^{1/2} \E\left[\big|N_n^{\yurev{ c'_X}}(X'_t - X'_{t_{j-1}})\big|^{2K}\right]^{1/2}dt
\nn\\
&\lesssim \max_{j=1,\dots, [n T_n]} \sup_{t_{j-1}\le t\le t_j}\E\left[|h^{-1/2}(X'_t - X'_{t_{j-1}})|^{2K}\right]^{1/2} 
\lesssim 1.
\label{hm:estimates+2}
\end{align}
The last estimate is due to the already verified Assumption \ref{hm:A_X}.2.
}
Let \yurev{$m_n=\max\{ i\ge 1 :\, i\tau\le hN_n\}=[hN_n/\tau]$}.
Then, $\yurev{m_n}\to\infty$ as $\yurev{m_n}\gtrsim n^{\ep''}$ by \eqref{hm:B_order}, and we have
\begin{equation}
\frac{1}{h N_n} \int_0^{h N_n} f(X'_t) dt = \frac{1}{h N_n} \int_0^{\tau \yurev{m_n}} f(X'_t) dt + \frac{1}{h N_n} \int_{\tau \yurev{m_n}}^{hN_n} f(X'_t) dt.
\nonumber
\end{equation}
The second term on the right-hand side can be bounded by a constant multiple of $(hN_n)^{-1}(1+|X'_t|^C) \lesssim n^{-\ep''}(1+|X'_t|^C)$. 
\hmrev{
Further letting $c'_X>0$ sufficiently small so that $n^{-\ep''} N_n^{c'_X}\lesssim 1$ if necessary,} 
we only need to look at the first term in the last display, say $V_n$.
By the periodicity of $X'_1$, we have
\begin{align}
V_n &= \frac{1}{h N_n} \sum_{m=1}^{\yurev{m_n}} \int_{(m-1)\tau}^{m\tau} f(X'_t) dt
\nn\\
&= \frac{\tau \yurev{m_n}}{h N_n} \frac{1}{\yurev{m_n}}\sum_{m=1}^{\yurev{m_n}} \frac{1}{\tau}\int_{0}^{\tau} f\left( X'_{1,(m-1)\tau + t},X'_{2,(m-1)\tau + t}\right) dt
\nn\\
&= \frac{\tau \yurev{m_n}}{h N_n} \frac{1}{\yurev{m_n}}\sum_{m=1}^{\yurev{m_n}} \frac{1}{\tau}\int_{0}^{\tau} f\left( X'_{1,t},X'_{2,(m-1)\tau + t}\right) dt
\nn\\
&=: \frac{\tau \yurev{m_n}}{h N_n} \, \frac{1}{\yurev{m_n}}\sum_{m=1}^{\yurev{m_n}} G_{f,m},\qquad\text{say.}
\nonumber
\end{align}
\hmrev{
Now we set
\begin{equation}
\psi_f(z) = \frac{1}{\tau}\int_{0}^{\tau} f( X'_{1,t}, z)dt.
\nonumber
\end{equation}
Then,
\begin{align}
& V_n - \int \psi_f(z) \pi_0(dz)
\nn\\
&= \left(\frac{\tau \yurev{m_n}}{h N_n} - 1\right) \, \frac{1}{\yurev{m_n}}\sum_{m=1}^{\yurev{m_n}} G_{f,m}
+\frac{1}{\yurev{m_n}}\sum_{m=1}^{\yurev{m_n}} \left(G_{f,m} - \E[G_{f,m}]\right)
\nn\\
&{}\qquad +\frac{1}{\yurev{m_n}}\sum_{m=1}^{\yurev{m_n}} \left(\E[G_{f,m}] - \int \psi_f(z) \pi_0(dz)\right)
\nn\\
&=: \overline{V}_{1,n} + \overline{V}_{2,n} + \overline{V}_{3,n},\qquad\text{say.}
\nonumber
\end{align}
Fix any $K\ge 2$.
Since $\sup_{m\ge 1}\E[|G_{f,m}|^K]<\infty$, $|\frac{\tau \yurev{m_n}}{h N_n}-1|=O((h N_n)^{-1})=O(n^{-\ep''})$ and $n^{-\ep''} N_n^{c'_X}\lesssim 1$, we get $\sup_n \E[|N_n^{c'_X}\overline{V}_{1,n}|^K]<\infty$.
We can handle the remaining two terms by the mixing property:
the condition \eqref{hm:rev1-2} implies that $X'_2$ is exponentially $\beta$-mixing (see \cite[Lemma 3.9]{Mas07}), hence also exponentially $\al$-mixing.
Since $G_{f,m}$ is $\sig(X'_{2,s}:\, s\in[(m-1)\tau,m\tau])$-measurable, there exists a constant $\mathfrak{c}_l=\mathfrak{c}_l(\tau)>0$ ($l=1,2$) for which
\begin{equation}
\al(k) := \sup_{i\ge 1}\sup_{A\in\sig(G_{f,l}:\,l\le i),\atop B\in\sig(G_{f,l}:\,l\ge i+k)}
\left|\pr[A\cap B]-\pr[A]\pr[B]\right| \le \mathfrak{c}_2 \exp(-\mathfrak{c}_1 k),\qquad k\in\mbbn.
\nonumber
\end{equation}
Then, a direct application of \cite[Lemma 4]{Yos11} yields $\sup_n \E[|N_n^{c'_X}\overline{V}_{2,n}|^K]<\infty$.
The remaining $\overline{V}_{3,n}$ can be treated similarly to the proof of \cite[Lemma 4.3]{Mas13as} as follows (see how to estimate $\Lam''_n(f;\theta)$ therein):
letting $g(z):=\sup_{|x'_1| \le \|X'_{1,\cdot}\|_\infty}|f(x'_1,z)|$ (hence $g(z)\lesssim 1+|z|^C$), we obtain from \eqref{hm:rev1-2} that
\begin{align}
& \left|\E[G_{f,m}] - \int \psi_f(z) \pi_0(dz)\right|
\nn\\
&= \left|\frac{1}{\tau}\int_{0}^{\tau} \iint f(X'_{1,t}, z)\,\left\{P_{(m-1)\tau+t}(x,dz) - \pi_0(dz)\right\} \pr^{X'_{2,0}}(dx) dt\right|
\nn\\
&\le \frac{1}{\tau}\int_{0}^{\tau} \int \left\| P_{(m-1)\tau+t}(x,\cdot) - \pi_0(\cdot)\right\|_g \pr^{X'_{2,0}}(dx) dt
\nn\\
&\lesssim \int g(x)\pr^{X'_{2,0}}(dx) \int_{0}^{\tau} \exp\left(-\kappa((m-1)\tau+t)\right)dt
\lesssim \exp\left(-\kappa\tau m\right),
\nonumber
\end{align}
\yurev{where $\pr^{X'_{2,0}}(\cdot)$ denotes the initial distribution of $X'_2$.}
It follows that $|N_n^{c'_X}\overline{V}_{3,n}| \lesssim N_n^{c'_X}M_n^{-1} \lesssim n^{-\ep''}N_n^{c'_X}\lesssim 1$, concluding Assumption \ref{hm:A_X}.3 (with $q'=q_2$).
}

Finally, it does not seem easy to give a general sufficient condition for the positive definiteness of $S_0$ in Assumption \ref{hm:A_X}.4, which is related to the positive definiteness at the true value $\theta_0$ of the Fisher-information matrix.


\section{Student quasi-likelihood analysis}
\label{hm:sec_student.LA}

Taking over the setting in Section \ref{hm:sec_Cauchy.QLA}, we now turn to the estimation of the degrees of freedom $\nu$ from $\{(X_{t_j},Y_{t_j})\}_{j=0}^{[nT_n]}$.
Suppose that Assumption \ref{hm:A_X} and \eqref{hm:B_order} hold, so that we have \eqref{hm:CQMLE.tpe} by Theorem \ref{hm:thm_CQLA}.
Throughout this section, we assume that $N_n$ is sufficiently large relative to $T_n$:
\begin{equation}
\frac{T_n}{N_n} = \frac{T_n}{[n B_n]} \to 0.
\label{T/N->0}
\end{equation}

\subsection{Construction and asymptotics}
\label{hm:sec_tQLA-add1}

Define the \textit{unit-time} residual sequence
\begin{equation}
\hep_i := \ses^{-1}\left( Y_i - Y_{i-1} - \mes\cdot(X_i - X_{i-1}) \right),
\nonumber
\end{equation}
for $i=1,\dots,[T_n]$.
Let
\begin{equation}
\ve_i:=J_i - J_{i-1}
\nonumber
\end{equation}
so that $\ve_1,\ve_2,\dots \sim \text{i.i.d.}~t_\nu$.
We will estimate $\nu$ based on the maximum-likelihood function as if $\hep_1,\dots,\hep_{[T_n]}$ are observed $t_\nu$-i.i.d. samples: we consider the explicit Student quasi-likelihood function
\begin{equation}
\mbbh_{2,n}(\nu) = \sumi \rho(\hep_i;\nu),
\label{hm:def_tQLF}
\end{equation}
where,
\hmrev{
with the notation \eqref{hm:def_f.nu.pdf},
\begin{equation}
\rho(\ve;\nu) := \log f(\ve;0,1,\nu)
=\log\left(
\frac{\Gamma(\frac{\nu+1}{2})}{\sqrt{\pi}\Gamma(\frac{\nu}{2})}
\left(1+\ve^2\right)^{-(\nu+1)/2}
\right).
\nonumber
\end{equation}
}
Then, we define the Student-$t$ quasi-likelihood estimator ($t$-QMLE for short) of $\nu$ by any element
\begin{equation}
\nes\in\argmax_{\nu\in\overline{\Theta}_\nu}\mbbh_{2,n}(\nu).
\nonumber
\end{equation}
We have
\begin{equation}
\mbbh_{2,n}(\nu) = 
\sumi\left( -\frac12 \log\pi 
+ \log\Gam\left(\frac{\nu+1}{2}\right) - \log\Gam\left(\frac{\nu}{2}\right)
- \frac{\nu+1}{2}\log\left( 1+ \hep_i^2\right)
\right).
\nonumber
\end{equation}
Let $\psi(x):=\p\log\Gam(x)$, the digamma function, and then $\psi_1:=\p\psi$, the trigamma function. 
By the integral representation $\p^m\psi(x)=(-1)^{m+1}\int_0^\infty s^{m}e^{-xs}(1-e^{-s})^{-1}ds$ for $m\in\mbbn$ (see \cite[6.4.1]{AbrSte92}), hence $\p\psi_1(x)<0$ for $x>0$.
From this fact and the last expression for $\mbbh_{2,n}(\nu)$, we obtain
\begin{equation}
- \p_\nu^2 \mbbh_{2,n}(\nu) = \frac{[T_n]}{4}\left(
\psi_1\left(\frac{\nu}{2}\right) - \psi_1\left(\frac{\nu+1}{2}\right)
\right) > 0
\label{hm:ep.hat-ep}
\end{equation}
for any $\nu>0$, hence $-\mbbh_{2,n}$ is a.s. convex on $(0,\infty)$.

Let
\begin{align}
\D_{\nu,n} &:= 
\frac{1}{\sqrt{T_n}} \sumi \p_\nu\rho(\ve_i;\nu_0),
\nn\\
\Gam_{\nu,0} &:= \frac{1}{4}\left(
\psi_1\left(\frac{\nu_0}{2}\right) - \psi_1\left(\frac{\nu_0+1}{2}\right)
\right).
\label{hm:def_Gam_nu}
\end{align}
We have $\Gam_{\nu,0}>0$ by \eqref{hm:ep.hat-ep}.
The next result shows that we can estimate $\nu_0$ directly as if we observe $(\ve_i)$.

\medskip

\begin{thm}
\label{hm:thm_tQLA}
Under 
\hmrrev{
the assumptions in Theorem \ref{hm:thm_CQLA} 
}
and \eqref{T/N->0}, we have the polynomial-type tail-probability estimate:
\begin{equation}
\forall L>0,\quad 
\sup_{r>0}\sup_n \pr\left[|\hat{u}_{\nu,n}| > r \right] r^L < \infty,
\label{hm:tQMLE.tpe}
\end{equation}
and the asymptotic normality:
\begin{equation}
\hat{u}_{\nu,n} := \sqrt{T_n}(\nes-\nu_0) = \Gam_{\nu,0}^{-1}\D_{\nu,n} + o_p(1) \cil N\left(0,\,\Gam_{\nu,0}^{-1}\right).
\label{hm:tQMLE.an}
\end{equation}
\end{thm}

The maximum-likelihood estimation of $\nu_0$ can become more unstable for a larger value of $\nu_0$.
The Fisher information $\Gam_{\nu,0}$ quickly decreases to $0$ as $\nu_0$ increases:
the asymptotic variance $\Gam_{\nu,0}^{-1}$ equals about $1.22$, $5.63$, $13.8$, $25.9$, and $41.9$ for $\nu_0=1,2,3,4$, and $5$, respectively.
The damping speed becomes even faster in the case of conventional parametrization.
See, for example, \cite[Section 2.2]{Har13} and the references therein.


\subsection{Joint asymptotic normality}
\label{hm:sec_jan}

Having Theorems \ref{hm:thm_CQLA} and \ref{hm:thm_tQLA} in hand, we consider the question of the joint asymptotic normality of $\hat{a}_n$ and $\nes$. 

Recall that we are given the underlying filtration $(\mcf_t)_{t\ge 0}$.
From the proofs of Theorems \ref{hm:thm_CQLA} and \ref{hm:thm_tQLA}, we can write $(\hat{u}_{a,n}, \hat{u}_{\nu,n})=M_n + o_p(1)$ with
\begin{equation}
M_n = \frac{1}{\sqrt{T_n}} \sumi G_i =\frac{1}{\sqrt{T_n}} \sumi \left(G_{a,i},\, G_{\nu,i}\right)
\nn
\end{equation}
for a martingale difference $G_{\nu,i}=\p_\nu\rho(\ve_i;\nu_0) \in \mbbr$ with respect to $(\mcf_i)_{i\in\mbbzp}$ and
\begin{align}
G_{a,i} &:= \sqrt{\frac{T_n}{N_n}} \, I(i\le [B_n]) \sum_{j\in A_i} \psi_{a,ij},
\nonumber
\end{align}
where $A_i:=\{k\in\mbbn:\, (i-1)\, n < k\le i \,n\}$ and $\psi_{a,ij}=\psi_{a,ij}(a_0) := -\p_a \left\{ \log\sig + \log\left(1+\ep_j(a)^2\right)\right\}|_{a=a_0}$.
It will be seen in the proof of Lemma \ref{hm:cqmle_lem3} that $T_n^{-1/2} \sumi \E[G_{a,i} |\mcf_{t_{j-1}}] = o_p(1)$.
Hence, $M_n$ is a partial sum of the approximate martingale difference array with respect to $(\mcf_i)_{i\in\mbbzp}$.

The Lyapunov condition holds for $M_n$.
To deduce the first-order asymptotic behavior of $M_n$, we need to look at the convergence in probability of the matrix-valued quadratic characteristic:
\begin{align}
[M]_n &:=
\frac{1}{T_n} \sumi 
\begin{pmatrix}
G_{a,i}^{\otimes 2} & G_{a,i}G_{\nu,i} \\ \text{sym.} & G_{\nu,i}^2
\end{pmatrix},
\nonumber
\end{align}
which explicitly depends on $\tz=(a_0,\nu_0)$; among others, we refer to \cite[Chapter VII.8]{Shi95} for the related basic facts.
By Theorems \ref{hm:thm_CQLA} and \ref{hm:thm_tQLA}, it remains to manage the cross-covariation part of $[M]_n$.

To that end, we further assume
\begin{equation}
\frac{B_n}{T_n} \to 0.
\label{B/T->0}
\end{equation}
By Lemmas \ref{hm:cqmle_lem2} and \ref{hm:tpe_lem2} in Section \ref{hm:sec_proof_thm1}, for any $K>0$,
\begin{equation}
\max_{1\le i\le [T_n]}\max_{1\le j\le [n T_n]} \E\left[|\psi_{a,ij}|^K + |G_{\nu,i}|^K\right] 
+ \max_{1\le i\le [T_n]}\E\left[\left| \frac{1}{\sqrt{n}}\sum_{j\in A_i} \psi_{a,ij} \right|^{K}\right] = O(1).
\nonumber
\end{equation}
This estimate leads to
\begin{align}
\frac{1}{T_n} \sumi G_{\nu,i}G_{a,i}
= \sqrt{\frac{B_n}{T_n}}\,\frac{1}{B_n}\sum_{i=1}^{[B_n]} G_{\nu,i} \sqrt{\frac{1}{n}} \sum_{j\in A_i} \psi_{a,ij}
= O_p\left(\sqrt{\frac{B_n}{T_n}}\right) = o_p(1).
\nonumber
\end{align}
Thus, the additional condition \eqref{B/T->0} approximately quantifies how much data we should discard to make the estimators $\hat{a}_n$ and $\nes$ independent.

Under Assumption \ref{hm:A_X},
\begin{align}
\hat{S}_n 
&:= \frac{1}{N_n} \sum_{j=1}^{N_n} \left(\frac{1}{h}\D_j X\right)^{\otimes 2}
\nn\\
&= \frac{1}{N_n} \sum_{j=1}^{N_n} \left\{\left(\frac{1}{h}\D_j X - X'_{t_{j-1}} \right) + X'_{t_{j-1}}\right\}^{\otimes 2}
\nn\\
&= \frac{1}{N_n} \sum_{j=1}^{N_n} (X'_{t_{j-1}})^{\otimes 2} + o_p(1) \cip S_0,\qquad n\to\infty,
\label{hm:S-hat}
\end{align}
Recalling the notation \eqref{hm:def_Gam_a} and \eqref{hm:def_Gam_nu}, we introduce
\begin{align}
\hat{\Gam}_{a,n} &:= \diag\left( \frac{1}{2\ses^2} \,\hat{S}_n , ~\frac{1}{2\ses^2} \right),
\nn\\
\hat{\Gam}_{\nu,n} &:= \frac{1}{4}\left(
\psi_1\left(\frac{\nes}{2}\right) - \psi_1\left(\frac{\nes+1}{2}\right)
\right).
\nonumber
\end{align}
Obviously we have $(\hat{\Gam}_{a,n},\hat{\Gam}_{a,n}) \cip (\Gam_{a,0},\Gam_{\nu,0})$, 
which combined with \eqref{hm:CQMLE.an} and \eqref{hm:tQMLE.an} leads to the joint asymptotic normality:

\medskip

\begin{thm}
\label{hm:thm_joint.AN}
Under 
\hmrrev{
the assumptions in Theorem \ref{hm:thm_CQLA}, 
}
\eqref{T/N->0}, and \eqref{B/T->0}, we have
\begin{equation}
\left( \hat{\Gam}_{a,n}^{1/2} \hat{u}_{a,n},\, \hat{\Gam}_{\nu,n}^{1/2}\hat{u}_{\nu,n}\right)
=
\diag
\left( \hat{\Gam}_{a,n}^{1/2},\, \hat{\Gam}_{\nu,n}^{1/2}\right) M_n + o_p(1) \cil N_{q+2}(0,I_{q+2}).
\label{hm:thm_joint.AN-1}
\end{equation}
\end{thm}



\begin{rem}[Asymptotic mixed normality at the first stage]
\normalfont
\label{hm:rem_AMN}
We have focused on a diverging $B_n$ to estimate $a=(\mu,\sig)$ at the first stage (recall \eqref{hm:B_order}).
Nevertheless, we could consider a constant $B_n$, say $B_n\equiv B\in\mbbn$ (then \eqref{B/T->0} is trivial), with which the CQMLE considered in Section \ref{hm:sec_Cauchy.QLA} is asymptotically mixed normal (MN):
\begin{equation}
\left(\hat{u}_{a,n},\, \hat{u}_{\nu,n}\right) \cil MN_{q+1}(0,\Gam_{a,0}^{-1}) \otimes N_1(0,\Gam_{\nu,0}^{-1}),
\nonumber
\end{equation}
now $\Gam_{a,0} = (2\sig_0^2)^{-1}\diag\left(S_0, 1\right)$ with $S_0=B^{-1}\int_0^B X^{\prime \otimes 2}_t dt$ being random.
Then, \eqref{hm:thm_joint.AN-1} remains valid by the proof of \eqref{hm:tQMLE.an} in Theorem \ref{hm:thm_tQLA}.
Without going into details, we give some brief remarks.
The asymptotic mixed normality is deduced as in \cite{Mas19spa} with a slight modification of the proof of the stable convergence in law of $\D_{a,n}$ (to handle the filtration-structure issue caused by a {\wp} $w$ driving $X'_2$: see \cite[Section 6.4.2]{Mas19spa}).
Because of the stability of the convergence in law, the conclusion 
of Theorem \ref{hm:thm_joint.AN} remains valid as it is with all the ``$N_n$'' therein replaced by ``$nB$''.
\end{rem}

\medskip

Our primary theoretical interest was to deduce the mighty convergence of $(\hat{u}_{a,n},\hat{u}_{\nu,n})$: not only the asymptotic normality but also the tail-probability estimate.
If we are only interested in deriving the asymptotic normality, there are several possible ways to significantly weaken Assumption \ref{hm:A_X}.
Here is a version with the essential boundedness of $X'$.

\medskip

\begin{assump}
\label{hm:A_X_AN}
For $X$, we can associate an $(\mcf_t)$-adapted process $t\mapsto X'_t$ having {\cadlag} sample paths, for which the following conditions hold:
\begin{enumerate}
\item $\ds{\sup_{t\ge 0}| X'_t | <\infty}$ a.s.
\item 
\hmrev{
The same condition as in Assumption \ref{hm:A_X}.3 holds except that \eqref{hm:A_X-3} is weakened to
\begin{equation}
\frac{1}{N_n} \sum_{j=1}^{N_n} f(X'_{t_{j-1}}) \cip \int \psi_f (z)\pi_0(dz).
\nn
\end{equation}
}

\item 
\hmrev{
$\ds{\frac{1}{N_n}\sum_{j=1}^{N_n} \E\left.\left[\left|\frac1h \left( \D_j X - h X'_{t_{j-1}}\right)\right|^2 \right| \mcf_{t_{j-1}}\right] = o_p(1)}$.
}

\item Assumption \ref{hm:A_X}.4 holds.
\item 
\hmrev{
$\ds{\frac{1}{\sqrt{N_n}}\sumjj \E\left.\left[ \frac1h \left(\D_j X  - h X'_{t_{j-1}}\right)
\,\frac{\p\phi_1}{\phi_1}\left(\frac{\D_j J}{h}\right) \right| \mcf_{t_{j-1}}\right] = o_p(1)}$
}
\end{enumerate}
\end{assump}

\medskip

\begin{thm}
\label{hm:thm_CQMLE.AN}
Let $\nu_0>1$. Suppose that 
\hmrev{\eqref{hm:B_order},} \eqref{T/N->0}, \eqref{B/T->0}, and Assumption \ref{hm:A_X_AN} hold.
Then, we have the joint asymptotic normality \eqref{hm:thm_joint.AN-1}.
\end{thm}

\subsection{Student-L\'{e}vy driven Markov process and Euler approximation}
\label{hm:sec_SDE}

We can handle a class of L\'{e}vy driven stochastic differential equation for $X_2$ in the context of Theorem \ref{hm:thm_CQMLE.AN}.
Consider a sample $(Y_{t_j})_{j=0}^{[n T_n]}$ from a solution to the Markov process ($Y_0=0$) described by the model
\begin{equation}
Y_t = \mu\cdot \int_0^t b(Y_s) ds + \sig J_t,
\nonumber
\end{equation}
where $\nu_0>1$ and $b:\,\mbbr\to\mbbr^q$ is a known measurable function.
We associate $X_t=\int_0^t b(Y_s)ds$ with $X'_t := b(Y_t)$.
So far it is assumed that $\D_j X = \int_{t_{j-1}}^{t_j}b(Y_s)ds$ is observable, which may seem unnatural in the present context. In this section, we will establish a variant of the asymptotic normality \eqref{hm:thm_joint.AN-1} when only $(Y_{t_j})_{j=0}^{[n T_n]}$ 
is observable.
To this end, we assume the following conditions, \hmrev{which will guarantee the ergodicity of $Y$}:

\medskip

\begin{assump}
\label{hm:assump_SDE}
The function $b$ is continuously differentiable and satisfies that
\begin{equation}
\sup_{y\in\mbbr}(|b(y)|\vee |\p_y b(y)|) < \infty 
\qquad \text{and}\qquad 
\hmrev{
\limsup_{|y|\to\infty}y\,\mu_0 \cdot b(y) = -\infty.
}
\nonumber
\end{equation}
\end{assump}

Let $b_{j-1}:=b(Y_{t_{j-1}})$ and introduce the following variant of $\mbbh_{1,n}(a)$:
\begin{align}
\tilde{\mbbh}_{1,n}(a) 
&:= \sum_{j=1}^{N_n} \log\left\{\frac{1}{h\sig}\phi_1\left( \frac{\D_j Y - h \mu\cdot b_{j-1}}{h\sig} \right)\right\}
=: \yurev{ \sum_{j=1}^{N_n} \log\left\{\frac{1}{h\sig}\phi_1(\ep'_j(a))\right\}}.
\nonumber
\end{align}
Define $\tilde{a}_n=(\tilde{\mu}_n,\tilde{\sig}_n)$ by any $\tilde{a}_n\in\argmax \tilde{\mbbh}_{1,n}$.
We write
\begin{equation}
\tep_i = \tilde{\sig}^{-1}\bigg( Y_i - Y_{i-1} - h \tilde{\mu}_n \cdot \sum_{j\in A_i} b_{j-1}\bigg)
\nonumber
\end{equation}
for $i=1,\dots,[T_n]$ (recall the notation $A_i=\{k\in\mbbn:\, (i-1)\, n < k\le i \,n\}$).
Next, we introduce
\begin{equation}
\tilde{\mbbh}_{2,n}(\nu) := \sumi \rho(\tep_i;\nu)
\label{hm:def_tQLF-2}
\end{equation}
and $\tilde{\nu}_n\in\argmax \tilde{\mbbh}_{2,n}$;
recall the definition $\rho(\ve;n)= \log f(\ve;0,1,\nu)$ (Section \ref{hm:sec_tQLA-add1}).
Let $\tilde{u}_{a,n} := \sqrt{N_n}(\tilde{a}_n - a_0)$, $\tilde{u}_{\nu,n} :=\sqrt{T_n}(\tilde{\nu}_n - \nu_0)$, and moreover
\begin{equation}
\tilde{\Gam}_{a,n} := \diag\left( \frac{1}{2\tilde{\sig}_n^2} \,\tilde{S}_n , ~\frac{1}{2\tilde{\sig}_n^2} \right),
\qquad 
\tilde{\Gam}_{\nu,n} := \frac{1}{4}\left(\psi_1\left(\frac{\tilde{\nu}_n}{2}\right) - \psi_1\left(\frac{\tilde{\nu}_n+1}{2}\right)\right),
\nonumber
\end{equation}
where $\tilde{S}_n := N_n^{-1} \sum_{j=1}^{N_n} b_{j-1}^{\otimes 2}$.

\medskip

\begin{thm}
\label{hm:thm_CQMLE.AN-SDE}
Let \hmrev{$\nu_0>2$} and suppose that \eqref{T/N->0}, \eqref{B/T->0}, and Assumption \ref{hm:assump_SDE} hold.
\hmrev{
Then, $Y$ is ergodic and admits a unique invariant distribution $\pi_{0,Y}(dy)$.
Moreover, we have
\begin{equation}
\left( \tilde{\Gam}_{a,n}^{1/2} \tilde{u}_{a,n},\, \tilde{\Gam}_{\nu,n}^{1/2} \tilde{u}_{\nu,n}\right)
=\diag\left( \tilde{\Gam}_{a,n}^{1/2},\, \tilde{\Gam}_{\nu,n}^{1/2}\right) M_n + o_p(1) \cil N_{q+2}(0,I_{q+2})
\nn
\end{equation}
provided that $\ds{S_0 := \int b(y)^{\otimes 2}\pi_{0,Y}(dy)}$ is positive definite.
}
\end{thm}

We see from the proof of Theorem \ref{hm:thm_CQMLE.AN-SDE} that rephrasing observed variables from $(X_{t_j})$ to $(X'_{t_j})$ when $X_t=\int_0^t X'_s ds$ is possible in a more general framework.

\section{Numerical experiment}
\label{sec_numerics}

In this section, we consider two deterministic and periodic regressors and the dynamics of the model have the following form:
\begin{equation}
Y_t =\mu_1 \cos\left(5t\right)+\mu_2 \sin\left(t\right) +\sigma J_t.
\label{Regr1}
\end{equation}

Since the Student-$t$ distribution is not closed under convolution as was stated in the introduction, we adopt an approximation method based on the inversion of the characteristic function for generating $J_h$ with general $h\neq 1$.
More specifically, our approximation is as follows:
First, on some grid $x_0,\dots, x_N$, we calculate the value of the density function of $J_h$ based on the following general inversion formula:
\begin{equation}
f\left(x\right) = \frac{1}{2\pi}\int_{-\infty}^{+\infty} e^{-iux}\varphi\left(u\right) \mbox{d}u,
\label{eq:INVCharLM}
\end{equation}
by using a numerical integration method, for example, Fast Fourier Transformation, Fourier Cosine expansion, and so on.
Next, we approximate the cumulative distribution function using the Left-Riemann summation computed on the grid $x_0,\dots, x_N$, therefore the cumulative distribution function $F\left(\cdot\right)$ for each $x_j$ on the grid is determined as follows:
\begin{equation}
\hat{F}\left(x_j\right) = \sum_{x_k<x_j}\hat{f}_{N}\left(x_k\right)\Delta x,
\end{equation}
where $\hat{f}_{N}\left(x_k\right)$ denotes the approximated value of $f(x_k)$.
Finally, we evaluate the cumulative distribution function at any $x\in \left( x_{j-1}, x_j \right)$, by interpolating linearly its value using the couples $\left(x_{j-1}, \hat{F}\left(x_{j-1}\right)\right)$ and $\left(x_{j}, \hat{F}\left(x_{j}\right)\right)$.
The random numbers can be obtained using the inversion sampling method.
Our simulation method and estimation procedure are available in \yuima package on R;
%
see the companion paper \cite{Hiroki2024Student} for details.

For the first simulation exercise, we set $T_n=400$, $B_n = 20$ and $h =1/200$. We simulate 1000 trajectories of the model with deterministic regressors in \eqref{Regr1} with $\mu_1 = 5$, $\mu_2 = -1$, $\sigma_0=3$ and $\nu = 1$. Estimating the parameters for each trajectory, we obtain the empirical distribution of the studentized estimates.
Figure \ref{fig:Case_h_0005_T_400_B_20_FFT_up_100_N_2_16_Ngrid1_5milion} shows the simulated distribution of studentized estimates. 
Notably, all histograms demonstrate that the standard normal approximation adequately captures the behavior of model parameters. 
Favorable results are obtained for the regressors $\left(\mu_1, \mu_2\right)$ and the scale coefficient $\sigma_0$. 
However, a small upward bias appears for the degree of freedom $\nu$. This bias can be controlled by increasing the numerical precision and/or considering a larger value for $T_n$.

\begin{figure}[!htbp]
	\centering
		\includegraphics[width=1.00\textwidth]{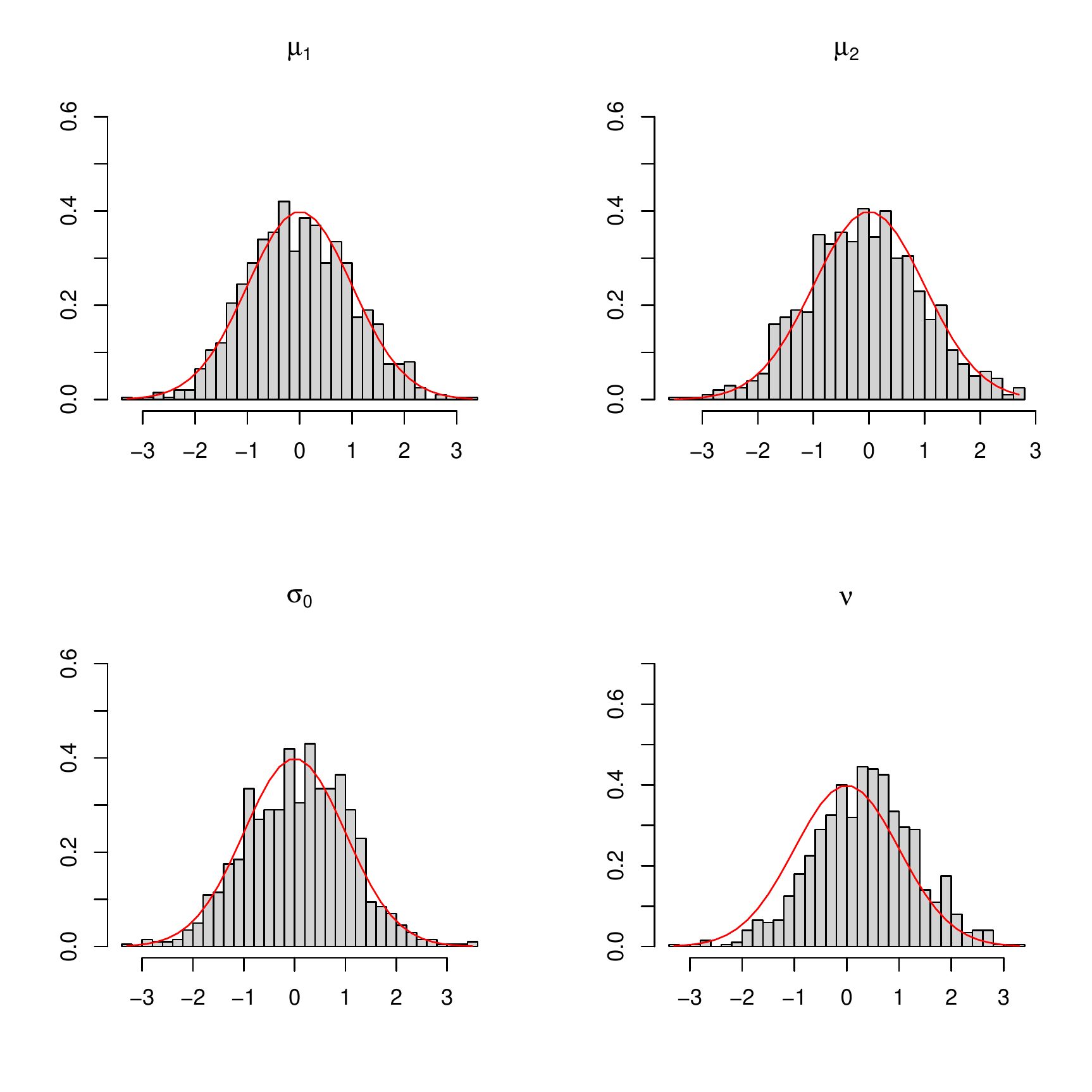}
	\caption{Simulated distribution of studentized parameters. In this case, we have $\mu_1=5$, $\mu_2=-1$, $\sigma_0 = 3$ and $\nu = 1$.}
	\label{fig:Case_h_0005_T_400_B_20_FFT_up_100_N_2_16_Ngrid1_5milion}
\end{figure}

In the second simulation exercise, we focus on the case of $\nu = 2$ and all other inputs remain unchanged. Figure \ref{fig:Case_h_1_over_300_T_400_B_20_FFT_30_N_2_rtp_16_Ngrid_10_rtp_6_nu_2} shows the histograms of the studentized estimators.
The comparison demonstrates a favorable agreement between the simulated density functions and the standard normal distribution for all model parameters. 
This comparison suggests that a larger value of the degree of freedom $\nu$ requires a smaller step-size value $h$ to ensure accurate approximations.
\begin{figure}[htbp]
	\centering
		\includegraphics[width=1.00\textwidth]{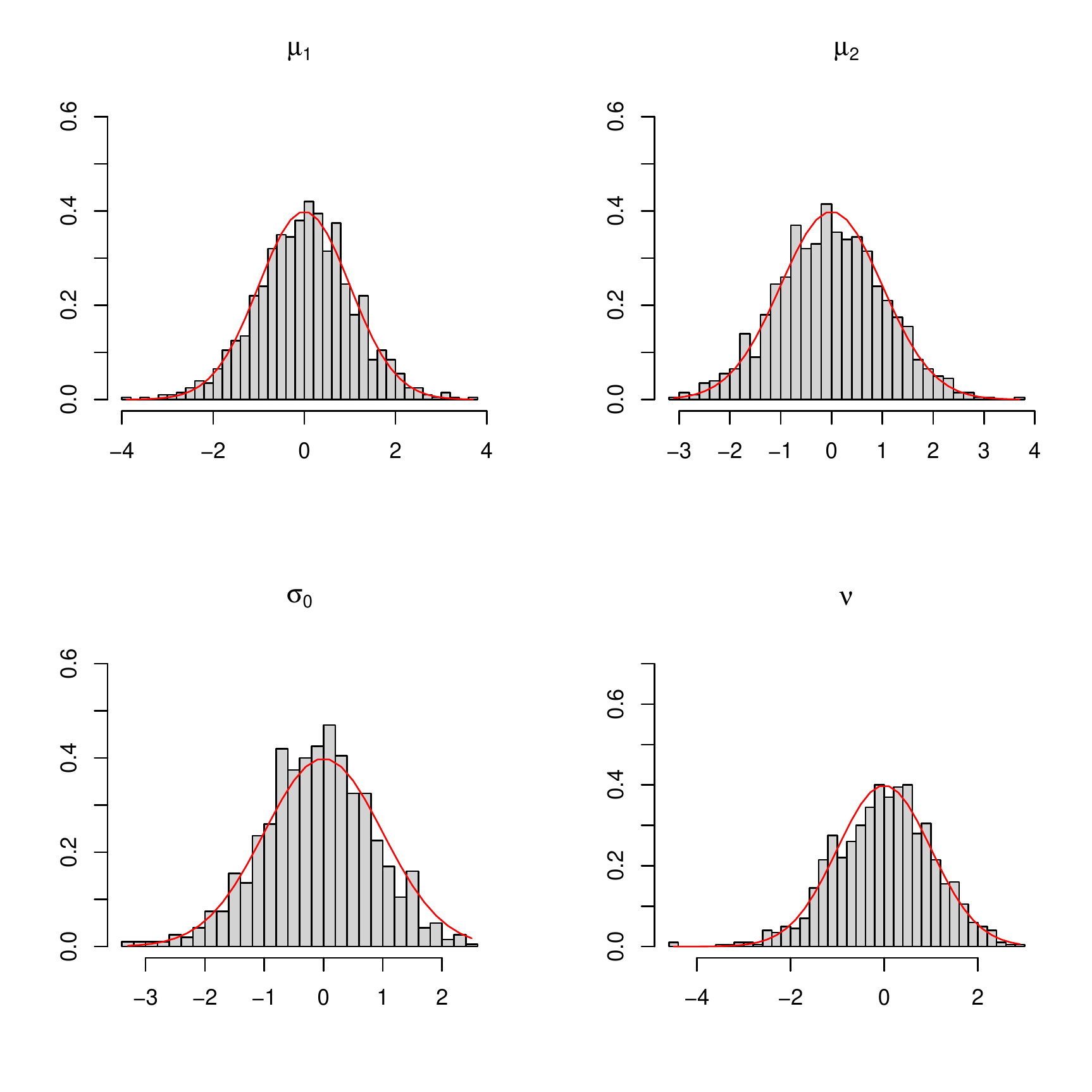}
	\caption{Simulated distribution of studentized parameters. In this case, we have $\mu_1=5$, $\mu_2=-1$, $\sigma_0 = 3$ and $\nu = 2$.}
	\label{fig:Case_h_1_over_300_T_400_B_20_FFT_30_N_2_rtp_16_Ngrid_10_rtp_6_nu_2}
\end{figure}

In conclusion, the analysis of estimator behavior through these simulation exercises provides valuable insights into their performance within the Student L\'evy Regression model. The results indicate satisfactory performance for most estimators, with minor biases observed for the degree of freedom parameter. Adjustments in numerical precision and step-time intervals can effectively control these biases. 
We refer to \cite{Hiroki2024Student} for further numerical experiments and real data analysis.

\section{Proofs}
\label{hm:sec_proofs}

\subsection{Proof of Theorem \ref{hm:thm_CQLA}}
\label{hm:sec_proof_thm1}

Let
\begin{align}
\mbby_{1,n}(a) &:= \frac{1}{N_n}\left( \mbbh_{1,n}(a) - \mbbh_{1,n}(a_0)\right), \nn\\
\Gam_{a,n} &:= -\frac{1}{N_n} \p_a^2\mbbh_{1,n}(a_0).
\nonumber
\end{align}
We will prove the following three lemmas 
\hmrrev{(recall \eqref{hm:def_new.Y1} for the definition of $\mbby_1(a)$)}.

\begin{lem}
\label{hm:cqmle_lem1}
Under \eqref{hm:B_order} and Assumption \ref{hm:A_X}, there exist 
constants $c_{0}>0$ and $c_{2}\in(0,1/2]$ such that for all $K>0$,
\begin{equation}
\sup_n \E\left[ \sup_{a}\big(N_n^{c_{2}}\left|\mbby_{1,n}(a)-\mbby_{1}(a)\right|\big)^K \right] < \infty.
\label{hm:QLA.thm1-3}
\end{equation}
\end{lem}

\begin{lem}
\label{hm:cqmle_lem2}
Under \eqref{hm:B_order} and Assumption \ref{hm:A_X}, there exist a constant $c_{1}\in(0,1/2)$ such that for every $K>0$,
\begin{align}
& \sup_n \E\left[|\D_{a,n}|^K\right] +
\sup_n \E \left[ \left(N_n^{c_{1}}\left|\Gam_{a,n} - \Gam_{a,0} \right|\right)^{K} \right] 
\nn\\
&{}\qquad + \sup_n \E\left[\sup_{a}\left| \frac{1}{N_n}\p_{a}^{3}\mbbh_{1,n}(a)\right|^K\right]
< \infty.
\label{hm:cqmle_lem2-1}
\end{align}
\end{lem}

\begin{lem}
\label{hm:cqmle_lem3}
Under \eqref{hm:B_order} and Assumption \ref{hm:A_X}, we have $\left(\Delta_{a,n},\, \Gam_{a,n}\right) \cil \big( \Gam_{a,0}^{1/2}\eta,\, \Gam_{a,0}\big)$, where $\eta \sim N_{q+1}(0,I_{q+1})$.
\end{lem}

\medskip

Theorem \ref{hm:thm_CQLA} follows from the above lemmas and \eqref{hm:QLA.thm1-4}:
with Lemmas \ref{hm:cqmle_lem1} and \ref{hm:cqmle_lem2}, we can establish the tail-probability estimate \eqref{hm:CQMLE.tpe} through \cite[Theorem 3(c)]{Yos11}; moreover, the asymptotic normality \eqref{hm:CQMLE.an} follows from the standard likelihood-asymptotics argument together with Lemmas \ref{hm:cqmle_lem3}.
We omit the details.

By \eqref{hm:ff_def_pre}, the probability density of $\mcl(h^{-1}J_h)$ is given by
\begin{align}
f_h(y;\nu) &= \frac{1}{2\pi}\int \cos(uy)\left\{\vp_{J_1,\nu}(h^{-1}u)\right\}^h
du
\nn\\
&= \left(\frac{2^{1-\nu/2} h^{-\nu/2}}{\Gam(\nu/2)} \right)^h
\frac{h}{\pi}\int_0^\infty \cos(uh x) u^{\nu h/2} \left(K_{\nu/2}(u)\right)^h du,
\label{hm:ff_def}
\end{align}
the existence of which is ensured by \cite[Proposition 1]{BerDon97} together with the locally Cauchy property $h^{-1}J_h \cil t_1$.
The next lemma quantifies the speed of the local Cauchy approximation in $L^1(dy)$ and serves as a basic tool for deriving limit theorems.

\begin{lem}
\label{hm:lem_loc.lim.thm}
For any $r>0$ and any measurable function $\zeta:\mbbr\to\mbbr$ such that
$|\zeta(y)| \lesssim 1+ \{\log(1+|y|)\}^K$ for some $K>0$, we have
\begin{equation}
\int \left| \zeta(y) \right| \left| f_h(y) - \phi_1(y) \right| dy \lesssim h^{1-r}.
\label{hm:lem_loc.lim.thm-1}
\end{equation}
In particular, under \eqref{hm:B_order} we have
\begin{equation}
\sqrt{N_n} \int \left| \zeta(y) \right| \left| f_h(y) - \phi_1(y) \right| dy \to 0.
\label{hm:lem_loc.lim.thm-2}
\end{equation}
\end{lem}

Related remarks on the estimate under the total variation distance can be found in \cite[Remark 2.1]{CleGlo20}.

\subsubsection{Proof of Lemma \ref{hm:cqmle_lem1}}
\label{hm:sec_proof_C.lem1}

\hmrev{
Let $\ep_j=\ep_{n,j}:=\ep_j(a_0)$, so that $\ep_1,\ep_2\dots\sim\text{i.i.d.}~\mcl(h^{-1}J_h)$, which admits the density $f_h(y)$ defined by \eqref{hm:ff_def};
this may be an abuse of notation (recall $\ve_i=J_i - J_{i-1}$ introduced in Section \ref{hm:sec_tQLA-add1}), but no confusion is likely to arise.
}
We have
\begin{align}
\ep_j(a) 
&= \frac{\sig_0}{\sig} \ep_j - \frac{1}{\sig} (\mu-\mu_0)\cdot X'_{t_{j-1}} - \frac{1}{\sig}(\mu-\mu_0)\cdot R_{n,j},
\nonumber
\end{align}
where $R_{n,j}:= h^{-1}(\D_j X  - h X'_{t_{j-1}})$.
Let $g_1(\ep):=\p_{\ep}\phi_1(\ep)/\phi_1(\ep)$ and recall the definition \eqref{hm:zeta_def} of $\zeta(y,x';\,a)$.
We have the decomposition $\mbby_{1,n}(a)=\mbby_{1,n}^\ast(a)+\del_{11,n}(a)$ where
\begin{align}
\mbby_{1,n}^\ast(a) &= \frac{1}{N_n}\sumjj \zeta(\ep_j,X'_{t_{j-1}};\,a),
\label{hm:def_del.Y.ast}\\
\del_{11,n}(a) &:= - 
\frac{1}{N_n}\sumjj
\int_0^1 g_1\left( \frac{\sig_0}{\sig}\ep_j - \frac{1}{\sig}(\mu-\mu_0)\cdot X'_{t_{j-1}}
-\frac{s}{\sig}(\mu-\mu_0)\cdot R_{n,j} \right) ds 
\nn\\
&{}\qquad \times \frac{1}{\sig}(\mu-\mu_0)\cdot R_{n,j}.
\label{hm:def_del11a}
\end{align}

Fix any $K>2$. 
By Assumption \ref{hm:A_X} and \eqref{hm:B_order}, we can pick a sufficiently small $c_2\in(0,\frac{c_X}{2-\ep'}\wedge\frac12]$ to obtain
\begin{equation}
\max_{j\le N_n}\E[|N_n^{c_2}R_{n,j}|^K] 
\lesssim \max_{j\le N_n}\E[|n^{c_2(2-\ep')}R_{n,j}|^K] \le \max_{j\le [n T_n]}\E[|n^{c_2(2-\ep')} R_{n,j}|^K] \to 0.
\nonumber
\end{equation}
Since $g_1$ is essentially bounded, we have
\begin{equation}
\E\left[\sup_a |N_n^{c_2} \del_{11,n}(a)|^K\right] \lesssim \max_{j\le N_n}\E[|N_n^{c_2}R_{n,j}|^K] 
\to 0.
\nonumber
\end{equation}
Write $\E^{j-1}[\cdot]=E[\cdot|\mcf_{t_{j-1}}]$. Then,
\begin{align}
\mbby_{1,n}^\ast(a) 
&= \frac{1}{\sqrt{N_n}} \sumjj \frac{1}{\sqrt{N_n}}\widetilde{\zeta}_j(a) + \frac{1}{N_n}\sumjj \E^{j-1}\big[\zeta(\ep_j,X'_{t_{j-1}};\,a)\big],
\label{hm:cqmle_lem1-p1}
\end{align}
where $\widetilde{\zeta}_j(a) := \zeta(\ep_j,X'_{t_{j-1}};\,a) - \E^{j-1}[\zeta(\ep_j,X'_{t_{j-1}};\,a)]$.
Obviously,
\begin{equation}
\sup_n \sup_a \left(\E\left[\left|\sumjj \frac{1}{\sqrt{N_n}} 
\widetilde{\zeta}_j(a)
\right|^K\right]
+ \E\left[\left|\sumjj \frac{1}{\sqrt{N_n}} 
\p_a \widetilde{\zeta}_j(a)
\right|^K\right] \right) < \infty.
\nonumber
\end{equation}
Since $\Theta$ is assumed to be a bounded convex domain, the Sobolev inequality \cite{AdaFou03} holds:
\begin{equation}
\sup_a |F(a)|^r \lesssim \int_{\Theta_\mu\times\Theta_\sig}\left( |F(a)|^r+|\p_a F(a)|^r \right) da
\nonumber
\end{equation}
for any $\mcc^1$-function $F$ and $r>\mathrm{dim}(a)=q+1$ (see \cite{AdaFou03} for details).
Applying this and Fubini's theorem, we conclude that
\begin{equation}
\sup_n \E\left[\sup_a \left| \sumjj \frac{1}{\sqrt{N_n}}
\widetilde{\zeta}_j(a)
\right|^K \right] < \infty,
\label{hm:cqmle_lem1-p2}
\end{equation}
followed by
\begin{equation}
\sup_n \E\left[ \sup_{a}\left(N_n^{c_{2}}\left|\mbby_{1,n}(a) - \frac{1}{N_n}\sumjj 
\E^{j-1}\big[\zeta(\ep_j,X'_{t_{j-1}};\,a)\big]
\right|\right)^K \right] < \infty.
\nn
\end{equation}

Observe that
\begin{align}
& \frac{1}{N_n}\sumjj 
\E^{j-1}\big[\zeta(\ep_j,X'_{t_{j-1}};\,a)\big]
\nn\\
&= \frac{1}{N_n}\sumjj 
\hmrrev{
f'(X'_{t_{j-1}};a)
}
+ \frac{1}{N_n}\sumjj \int \zeta(y,X'_{t_{j-1}};a) (f_h(y) - \phi_1(y))dy
\nn\\
&=: \mbby^{\ast\ast}_{1,n}(a) + \del_{12,n}(a).
\label{hm:cqmle_lem2-p1}
\end{align}
Note that $\sup_a |\zeta(y,X'_{t_{j-1}};a)| \lesssim 1+\log(1+|X'_{t_{j-1}}|) + \log(1+|y|) \lesssim \{1+\log(1+|X'_{t_{j-1}}|)\} \{1 + \log(1+|y|)\}$.
Also, $N_n^{c_2} h^{1-c}\lesssim n^{c_2(2-\ep')}n^{c-1} \to 0$ with $c_2,c>0$ small enough.
It follows from Lemma \ref{hm:lem_loc.lim.thm} and Assumption \ref{hm:A_X} that
\begin{align}
\E\left[\sup_a |N_n^{c_2} \del_{12,n}(a)|^K\right] 
&\lesssim \max_{j\le N_n}\E\left[\left| (1+|X'_{t_{j-1}}|)^C N_n^{c_2} h^{1-c} \right|^K \right]
\nn\\
&\lesssim (N_n^{c_2} h^{1-c})^K \sup_{0\le t\le h N_n}\E\left[(1+|X'_{t_{j-1}}|)^C\right] \to 0.
\nonumber
\end{align}
\hmrrev{
Moreover, since
$\max_{k\in\{0,1\}}\sup_a |\p_{a}^k f'(x';a)| \lesssim 1 + |x|^C$, Assumption \ref{hm:A_X}.3 ensures that
\begin{equation}
\max_{k\in\{0,1\}}\sup_{n} \sup_a \E\left[\left| N_n^{c_X'}\left(
\p_a^k \mbby^{\ast\ast}_{1,n}(a)
- \int \psi_{\p_a^k f'(\cdot;a)} (z)\pi_0(dz)\right)\right|^K \right] < \infty
\nn%
\end{equation}
for every $K>0$.
}

By piecing together what we have seen 
\hmrev{with the Sobolev inequality as before, we obtain the following estimate with a sufficiently small $c_2>0$:}
\begin{align}
& \sup_n \E\left[ \sup_{a}\big(N_n^{c_{2}}\left|\mbby_{1,n}(a)-\mbby_{1}(a)\right|\big)^K \right]
\nn\\
&\lesssim 1 + \sup_n \E\left[ \sup_{a}\big(N_n^{c_{2}}\left|\mbby^{\ast\ast}_{1,n}(a)-\mbby_{1}(a)\right|\big)^K \right]
\lesssim 1.
\label{hm:estimates+1}
\end{align}
\subsubsection{Proof of Lemma \ref{hm:cqmle_lem2}}
\label{hm:cqmle_lem2_proof}

We will look at the three terms in \eqref{hm:cqmle_lem2-1} separately.

\medskip

\paragraph{Quasi-score function}

Recall that $\ep_j=\ep_j(a_0) \sim \text{i.i.d.}~\mcl(h^{-1}J_h)$. Direct calculations give
\begin{align}
\D_{a,n}
&= -\sig_0^{-1}\frac{1}{\sqrt{N_n}}\sumjj 
\begin{pmatrix}
(h^{-1} \D_j X) g_1(\ep_j) \\
1+\ep_j g_1(\ep_j)
\end{pmatrix}
\nn\\
&= -\sig_0^{-1}
\frac{1}{\sqrt{N_n}}\sumjj 
\begin{pmatrix}
X'_{t_{j-1}}g_1(\ep_j)
\\
1+\ep_j g_1(\ep_j)
\end{pmatrix}
-\sig_0^{-1}
\frac{1}{N_n}\sumjj 
\begin{pmatrix}
\sqrt{N_n}R_{n,j}\,g_1(\ep_j)
\\
0
\end{pmatrix}.
\label{hm:Da-p1}
\end{align}

The first term in \eqref{hm:Da-p1} equals
\begin{align}
& -\sig_0^{-1}
\frac{1}{\sqrt{N_n}}\sumjj 
\begin{pmatrix}
X'_{t_{j-1}}(g_1(\ep_j)-\E[g_1(\ep_j)])
\\
1+\ep_j g_1(\ep_j) - \E[1+\ep_j g_1(\ep_j)]
\end{pmatrix}
\nn\\
&{}\qquad 
\yurev{-}\sig_0^{-1}
\frac{1}{N_n}\sumjj 
\sqrt{N_n}
\begin{pmatrix}
X'_{t_{j-1}}(\E[g_1(\ep_j)] - \int g_1(y)\phi_1(y)dy)
\\
\E[1+\ep_j g_1(\ep_j)]- \int (1+y \p g_1(y))\phi_1(y)dy
\end{pmatrix}.
\nn\\
&=: r'_{1,n} + r''_{1,n}.
\label{hm:Da-p2}
\end{align}
The Burkholder inequality gives the $L^K(\pr)$-boundedness of $(r'_{1,n})_n$ for any $K>0$.
By Lemma \ref{hm:lem_loc.lim.thm}, the summand in $r''_{1,n}$ can be bounded in absolute value by a constant multiple of $(1+|X'_{t_{j-1}}|) \sqrt{N_n}h^{1-c} \lesssim (1+|X'_{t_{j-1}}|) n^{c-\ep'/2}$. 
The sequence $(r''_{1,n})_n$ is also $L^K(\pr)$-bounded since $c>0$ can be taken arbitrarily small so that $c\le\ep'/2$.

For the second term in \eqref{hm:Da-p1}, by the compensation and the Burkholder inequality it suffices to show that for any $K\ge 2$,
\begin{equation}
\frac{1}{N_n} \sumjj 
\E\left[\left|
\sqrt{N_n}\E^{j-1}[R_{n,j}\,g_1(\ep_j)]
\right|^K\right] = o(1).
\nonumber
\end{equation}
This is ensured by the last convergence in Assumption \ref{hm:A_X}.2.

\medskip

\paragraph{Quasi-observed information}

The components of $\Gam_{a,n}$ consists of
\begin{align}
-\frac{1}{N_n}\p_\mu^2\mbbh_{1,n}(a_0)
&= -\sig_0^{-2}\frac{1}{N_n}\sumjj (h^{-1}\D_j X)^{\otimes 2} \p g_1(\ep_j),
\nn\\
-\frac{1}{N_n}\p_\sig^2\mbbh_{1,n}(a_0)
&= -\sig_0^{-2}\frac{1}{N_n}\sumjj \left(1+2\ep_j g_1(\ep_j) + \ep_j^2 \p g_1(\ep_j)\right),
\nn\\
-\frac{1}{N_n}\p_\mu\p_\sig\mbbh_{1,n}(a_0)
&= -\sig_0^{-2}\frac{1}{N_n}\sumjj (h^{-1}\D_j X) \left(g_1(\ep_j) - \ep_j \p g_1(\ep_j)\right).
\nonumber
\end{align}
To conclude that $\sup_n \E [ (N_n^{c_{1}}|\Gam_{a,n} - \Gam_{a,0} |)^{K} ]<\infty$, we note the following three preliminary steps (we can take $c_1>0$ as small as we want):
\begin{itemize}
\item First, we replace $h^{-1}\D_j X$ in the summands by $R_{n,j}+X'_{t_{j-1}}$; this together with Assumption \ref{hm:A_X} enables us to replace $h^{-1}\D_j X$ in the summands by $X'_{t_{j-1}}$;
\item Second, we extract the martingale terms by replacing the parts of the form $\eta(\ep_j)$ by $\eta(\ep_j) - \E[\eta(\ep_j)]$ 
and then apply the Burkholder inequality to the latter parts;
\item Third, we apply Lemma \ref{hm:lem_loc.lim.thm} to the expressions $\E[\eta(\ep_j)]$ with the facts $-\int \p g_1(y)\phi_1(y)dy = 1/2$, 
$-\int (1+ 2y g_1(y) + y^2\p g_1(y))\phi_1(y)dy=1/2$, and $\int (g_1(y) - y \p g_1(y))\phi_1(y)dy=0$.
\end{itemize}
It remains to note that, as in the last paragraph in the proof of Lemma \ref{hm:cqmle_lem1}, Assumption \ref{hm:A_X}.3 ensures that
\hmrrev{
(recall that $S_0 = \int \psi_f (z)\pi_0(dz)$ for $f(x')=x^{\prime \,\otimes 2}$: Assumption \ref{hm:A_X}.4)
}
\begin{align}
\sup_n \E \left[ \left( N_n^{c_{1}}\left|\frac{1}{N_n}\sumjj X_{t_{j-1}}^{\prime\, \otimes 2} - S_0\right| \right)^K\right] < \infty.
\nonumber
\end{align}

\medskip

\paragraph{Third-order derivatives}

We have
\begin{equation}
\frac{1}{N_n}\sup_{a}|\p_{a}^{3}\mbbh_{1,n}(a)| \lesssim \frac{1}{N_n}\sumjj \left( 1 + |h^{-1} \D_j X| \right),
\nonumber
\end{equation}
hence $\sup_n \E[\sup_{a}| N_n^{-1}\p_{a}^{3}\mbbh_{1,n}(a)|^K]<\infty$.

\subsubsection{Proof of Lemma \ref{hm:cqmle_lem3}}
\label{hm:sec_aAN_lem_proof}

We have $\D_{a,n} = r'_{1,n} + o_p(1)$ from the arguments in Section \ref{hm:cqmle_lem2_proof} (with $c<\ep'/2$).
Write $r'_{1,n}=\sumjj \zeta_{n,j}$. Then, the $(\mcf_{t_{j}})$-martingale difference array $(\zeta_{n,j})_{j\le [nT_n]}$ satisfies 
the Lyapunov condition ($\sumjj \E^{j-1}[|\zeta_{n,j}|^{\del+2}] = o_p(1)$ for $\del>0$).
Also, by using Lemma \ref{hm:lem_loc.lim.thm} as in Section \ref{hm:cqmle_lem2_proof} together with the facts $\int g_1(y)^2\phi_1(y)dy=\int (1+ y g_1(y))^2\phi_1(y)dy=1/2$ and $\int g_1(y)(1+ y g_1(y))\phi_1(y)dy=0$, we obtain
\begin{align}
& \frac{1}{N_n}\sumjj \E^{j-1}\left[\zeta_{n,j}^{\otimes 2}\right]
\nn\\
&= \sig_0^{-2}\frac{1}{N_n}\sumjj 
\begin{pmatrix}
(h^{-1} \D_j X)^{\otimes 2}g_1(\ep_j)^2 & \text{sym.} \\
(h^{-1} \D_j X) g_1(\ep_j)(1+ \ep_j g_1(\ep_j)) & (1+ \ep_j g_1(\ep_j))^2
\end{pmatrix}
+ o_p(1)
\nn\\
&= \sig_0^{-2}\frac{1}{N_n}\sumjj 
\begin{pmatrix}
X^{\prime \otimes 2}_{t_{j-1}}\int g_1(y)^2\phi_1(y)dy & 0 \\
0 & \int (1+ y g_1(y))^2\phi_1(y)dy
\end{pmatrix}
+ o_p(1)
\nn\\
&= \Gam_{a,0} + o_p(1).
\nonumber
\end{align}
It follows that $\D_{a,n} \cil N_{q+1}(0,\Gam_{a,0})$ by the central limit theorem for martingale difference arrays \cite[Chapter VII.8]{Shi95}.
The convergence $\Gam_{a,n} \cip \Gam_{a,0}$ is automatic by the argument in Section \ref{hm:cqmle_lem2}; indeed, it holds almost surely by Borel-Cantelli lemma.

\subsubsection{Proof of Lemma \ref{hm:lem_loc.lim.thm}}

As in \cite[Example 2.7]{Mas19spa}, our proof is based on the explicit form of the characteristic function \eqref{hm:t_nu(0,1)-CF}.
We will proceed similarly to the proof of \cite[Lemma 2.2]{Mas19spa}, with omitting some details of the calculation.

\hmrev{
We write $g(z)$ for the {\ld} of $\mcl(J_1)$ and $g_0(z):=\pi^{-1}z^{-2}$ for the {\ld} of the standard Cauchy distribution $t_1$.
Let $\rho(z):=g(z)/g_0(z) - 1$ so that $g(z)=g_0(z)(1+\rho(z))$; the function $\rho$ quantifies how much the $\mcl(J_1)$ is far from $t_1$ around the origin.
Then, by \cite[Proposition 2.18]{Rai00} we have
\begin{equation}
\rho(z) = \frac{\pi}{4}(1-\nu_0) |z| + o(|z|), \qquad |z|\to 0.
\label{hm:J.Ld}
\end{equation}
}
Let $\vp_h(u)$ and $\vp_0(u)$ denote the characteristic functions of $\mcl(h^{-1}J_h)$ and $t_1$, respectively; we have $\vp_0(u) = \exp\{\int (\cos(uy)-1) g_0(y)dy\}=e^{-|u|}$ and already noted in \eqref{hm:Cauchy.approx_CF} that
\begin{equation}
\vp_h(u) = \left(\frac{2^{1-\nu/2}}{\Gam(\nu/2)}\right)^h \left(|u/h|^{\nu/2} K_{\nu/2}(|u/h|)\right)^h \to \vp_0(u),
\quad h\to 0.
\nonumber
\end{equation}

By straightforward computations, we obtain $\vp_h(u)=\exp\{\int(\cos(uy)-1)h^2 g(hy)dy\}=\vp_0(u)\exp\{\chi_h(u)\}$ where 
\begin{equation}
\chi_h(u) := \int (\cos(uy)-1) g_0(y) \rho(hy)dy.
\nonumber
\end{equation}
Pick a $\del\in(0,1)$ for which $\sup_{|y|\le \del}|\rho(y)|\le 1/2$.
By \cite[Theorem 25.3]{Sat99} we know the equivalence
\begin{equation}
\int_{|z|>1}|z|^{r}g(z)dz <\infty \iff \E[|J_1|^r]<\infty\qquad (r>0).
\nonumber
\end{equation}
The function $\rho$ may be unbounded for $|z|\to\infty$ when $\nu_0\in(0,1)$.
Now observe that
\begin{align}
|\chi_h(u)| &\le \int (1-\cos(uy)) g_0(y) |\rho(hy)| dy
\nn\\
&= \int_{|y|\le \del/h} (1-\cos(uy)) g_0(y) |\rho(hy)| dy + \int_{|y|> \del/h} (1-\cos(uy)) g_0(y) |\rho(hy)| dy
\label{hm:llt-p2}\\
&\le \frac12\int_{|y|\le \del/h} (1-\cos(uy)) g_0(y)dy + \frac{2}{\pi^2} \int_{|y|>\del/h} y^{-2} |\rho(hy)| dy.
\label{hm:llt-p1}
\end{align}
\hmrev{
The first term in \eqref{hm:llt-p1} can be bounded by $-2^{-1}\int (\cos(uy)-1) g_0(y)dy = |u|/2$.
}
The integral in the second term in \eqref{hm:llt-p1} equals $h \int_{|y|>\del} y^{-2} |\rho(y)| dy$, which is $O(h)$ if $\rho$ is bounded.
When $\rho$ is unbounded, we have $\rho(y)\to \infty$ as $|y|\to\infty$ hence there is a $c_\rho>\del$ for which $\rho(y)\ge 0$ for $|y|\ge c_\rho$.
In this case, $\int_{|y|>\del} y^{-2} |\rho(y)| dy = \int_{|y|\in[\del,c_\rho)} y^{-2} |\rho(y)| dy + \int_{|y|>c_\rho} y^{-2} \rho(y) dy
\lesssim \int_{|y|\in[\del,c_\rho)} y^{-2} dy + \pi \int_{|y|>c_\rho} (g(y) - g_0(y)) dy < \infty$, so that the integral in the second term in \eqref{hm:llt-p1} is $O(h)$.
It follows that $|\chi_h(u)|\le |u|/2 + Ch$.
Also, the first term in \eqref{hm:llt-p2} can be bounded by a constant multiple of $h\{(1+u^2) + \log(1/h)\}$:
this can be seen by dividing the domain of integration 
into $\{y:\,|y|\le \del/h,\,|y|>c\}$ and $\{y:\,|y|\le \del/h,\, |y|\le c\}$, and then proceeding as before with using the fact $|1-\cos x| \lesssim x^2$ for the latter subdomain.

Following the same line as in the proof of \cite[Lemma 5.1]{Mas19spa}, we can conclude that for any $H>0$ and $\xi\in(0,1)$,
\begin{equation}
\int_0^\infty |\vp_h(u)-\vp_0(u)|du \le \int_0^\infty (u^{-\xi}\vee u^{H})|\vp_h(u)-\vp_0(u)|du \lesssim h \log(1/h).
\label{hm:llt-p6}
\end{equation}
This combined with the Fourier inversion formula yields
\begin{equation}
\sup_{y\in\mbbr}\left| f_h(y) - \phi_1(y) \right| \lesssim h \log(1/h).
\label{hm:llt-p3}
\end{equation}
As in the proof of \cite[Eq.(5.1)]{Mas19spa}, we can derive the estimate
\begin{equation}
\sup_{h\in(0,1]}\sup_{M>0} M^{(1\wedge \nu_0) - \kappa} \int_{|y|>M} |y|^\kappa f_h(y) dy < \infty
\nonumber
\end{equation}
valid for any $\kappa\in(0,1\wedge \nu_0)$, considering the cases $\nu_0\ge 1$ and $\nu_0\in(0,1)$ (for which $\rho$ is unbounded) separately as before.

Let $(b_n)$ be a positive real sequence such that $b_n\uparrow \infty$ as $n\to\infty$.
To deduce \eqref{hm:lem_loc.lim.thm-1}, we proceed with estimating $\del_n' := \int_{|y|>b_n} |\zeta(y)|\left| f_h(y) - \phi_1(y) \right| dy$ and $\del''_n:=\int_{|y|\le b_n} |\zeta(y)|\left| f_h(y) - \phi_1(y) \right| dy$ separately.
Since the density $f_h$ is bounded continuous uniformly in $h\in(0,1]$, we have 
$\sup_{h\in(0,1]} f_h(y) \lesssim |y|^{-(1+1\wedge \nu_0)+\ep_0}$ for any $|y|\ge 1$ and for any small $\ep_0>0$; recall \cite[Theorem 25.3]{Sat99} already mentioned above.
Since $\zeta(y)$ is of at most logarithmic-power growth order, there exists a sufficiently small 
\hmrrev{$\ep_1\in(0,1\wedge\nu_0)$ } for which 
\begin{align}
\del'_n &\le \int_{|y|>b_n} |\zeta(y)|f_h(y) dy + \int_{|y|>b_n} |\zeta(y)|\phi_1(y) dy 
\nn\\
&\lesssim 
\hmrrev{b_n^{-1\wedge\nu_0+\ep_1}} + b_n ^{-1+\ep_1} 
\lesssim \hmrrev{b_n^{-1\wedge\nu_0+\ep_1}}.
\label{hm:llt-p7}
\end{align}
\hmrrev{
Turning to $\del''_n$, we will show that $\del''_n \lesssim b_n^{c''} h\log(1/h)$ for arbitrarily small $c''>0$.
}
Let $\psi_h(u):=\log \vp_h(u)$ and $\psi_0(u):=\log \vp_0(u)$.
Fix $c\in(0,1]$ arbitrarily.
By the exactly same procedure as in \cite[Eq.(5.4)]{Mas19spa}, we obtain
\begin{equation}
\del''_n \lesssim (1+ b_n^{\ep_2}) \int_{|y|\le b_n}\left| f_h(y) - \phi_1(y) \right| dy
\lesssim (1+ b_n^{\ep_2})(\del''_{1,n}(c) + \del''_{2,n}(c)),
\nonumber
\end{equation}
where $\ep_2>0$ can be taken arbitrarily small 
and where
\begin{align}
\del''_{1,n}(c) &:= b_n^c \int_0^\infty u^{c}|\vp_h(u) - \vp_0(u)| |\p_u\psi_h(u)| du,
\label{hm:llt-p4} \\
\del''_{2,n}(c) &:= b_n^c \int_0^\infty u^{c}|\p_u\psi_h(u) - \p_u\psi_0(u)| \vp_0(u) du.
\label{hm:llt-p5}
\end{align}
We will take a closer look at these quantities through the specific form of $\vp_h(u)$.

Put $s=\nu_0/2$ in the sequel.
We have $\psi_h(u) = h C_s + sh\log u + sh \log(1/h) + h\log K_s(u/h)$ for $u>0$ with some constant $C_s$ only depending on $s$.
Since $\p_z K_\al(z)=-K_{\al-1}(z)-(\al/z)K_\al(z)$, we obtain the expression $\p_u\psi_h(u) - \p_u\psi_0(u) = 1- K_{s-1}(u/h)/K_s(u/h) =:\xi_s(u)$.
The function $x\mapsto x\xi_s(x)$ is smooth on $(0,\infty)$ and we can deduce that $\sup_{x>0}|x\xi_s(x)|<\infty$ as follows:
\begin{itemize}
\item we have $|x\xi_s(x)|\lesssim x^{1\wedge \nu_0}$ around the origin, by using the property $K_{r'}(x)\sim \Gam(|r'|)2^{|r'|-1}x^{-|r'|}$ for $x\to 0$ ($r'\ne 0$);
\item moreover, $x\xi_s(x) = x(\frac{1}{8x} + O(x^{-1}))=O(1)$ for $x\uparrow \infty$ since $K_s(x)\sim \sqrt{\frac{\pi}{2x}}\, (1+\frac{4s^2 - 1}{8x} + O(x^{-2}))$ for $x\uparrow \infty$.
\end{itemize}
Thus, we have obtained $|\p_u\psi_h(u) + 1| = |\xi_s(u/h)|\lesssim h/u$ and $|\p_u\psi_h(u)| \lesssim 1+h/u$ ($u>0$).
Substituting these estimates and \eqref{hm:llt-p6} into \eqref{hm:llt-p4} and \eqref{hm:llt-p5}, we obtain 
\hmrrev{$\del''_{1,n}(c)\lesssim b_n^{c} h\log(1/h)$ and $\del''_{2,n}(c)\lesssim b_n^{c} h$.
Hence it follows that $\del''_n \lesssim b_n^{c+\ep_2} h\log(1/h)$ where we can set $c+\ep_2>0$ arbitrarily small as was desired.}

Combining \eqref{hm:llt-p7} and the last estimate of $\del''_n$ now gives $\int |\zeta(y)|\left| f_h(y) - \phi_1(y) \right|dy \lesssim 
\hmrrev{b_n^{-(1\wedge\nu_0)+\ep_1}}
 + b_n^{\ep_2+c} h\log(1/h)$.
This upper bound is minimized for 
$b_n = (h\log(1/h))^{-1/\chi}$ with $\chi := \hmrrev{1\wedge \nu_0 + c+\ep_2 - \ep_1}$, with which we get
\begin{align}
\int \left| f_h(y) - \phi_1(y) \right| dy 
&\lesssim \left(h\log(1/h)\right)^{1 - \frac{c+\ep_2}{\chi}}.
\nonumber
\end{align}
Given any small $r>0$, we can take all of $\ep_1,\ep_2,c>0$ small enough to conclude \eqref{hm:lem_loc.lim.thm-1}.
Now \eqref{hm:lem_loc.lim.thm-2} is trivial under \eqref{hm:B_order}.

\subsection{Proof of Theorem \ref{hm:thm_tQLA}}

\subsubsection{Proof of (\ref{hm:tQMLE.tpe}): Tail-probability estimate}
\label{hm:sec_t.tpe}

This section aims to prove \eqref{hm:tQMLE.tpe} through \cite[Theorem 3(c)]{Yos11}.
Define
\begin{align}
\mbby_{2,n}(\nu) &= \frac{1}{T_n}\left( \mbbh_{2,n}(\nu) - \mbbh_{2,n}(\nu_0) \right),
\label{hm:def_Y2n}
\\
\mbby_{2}(\nu) &= \E[\rho(\ve_1;\nu) - \rho(\ve_1;\nu_0)]
= \int \log\left(\frac{f(\ve;\nu)}{f(\ve;\nu_0)}\right) f(\ve;\nu_0) d\ve.
\label{hm:def_Y2}
\end{align}

\begin{lem}
\label{hm:t_tpe_lem1}
~
\begin{enumerate}
\item There exists a positive constant $c'_{0}>0$ such that $\mbby_{2}(\nu) \le -c'_{0}|\nu-\nu_0|^{2}$ for every $\nu\in\Theta_\nu$.
\item There exists a constant $c'_2\in(0,1/2]$ such that for every $K>0$,
\begin{equation}
\sup_n \E\left[
\sup_\nu \left|T_n^{c'_{2}}\left(\mbby_{2,n}(\nu) - \mbby_{2}(\nu)\right)\right|^K
\right] < \infty.
\nonumber
\end{equation}
In particular, $\ds{\sup_{\nu}\left|\mbby_{2,n}(\nu)-\mbby_{2}(\nu)\right| \cip 0}$.
\item The consistency holds: $\nes\cip\nu_0$.
\end{enumerate}
\end{lem}

\begin{proof}
The proof of 1 is similar to the case of $\mbby_{1}(a)$ in Section \ref{hm:sec_Cauchy.QLA}, and the consistency 3 is an obvious consequence of 1 and 2.

To prove 2, \yurev{we fix $K>0$ in the rest of this proof}.
We will repeatedly use the following estimate:
\begin{align}
& \sup_{\ve}\sup_{\nu}
\bigg(
\sup_{k\ge 0}\sup_{l\ge 2}\big|\p_\ve^k \p_\nu^l \rho(\ve;\nu)\big|
+\frac{\max_{l\in\{0,1\}}\big|\p_\nu^l \rho(\ve;\nu)\big|}{1+\log(1+\ve^2)}
\nn\\
&{}\qquad
+\sup_{k\ge 1}\max_{l\in\{0,1\}}(1+|\ve|^k)\big|\p_\ve^k \p_\nu^l \rho(\ve;\nu)\big|
\bigg) < \infty.
\label{hm:rho.prop-1}
\end{align}
Since $\ve_1,\dots,\ve_{[T_n]}$ are i.i.d., we have the moment estimate
\begin{equation}
\sup_n \E\left[
\left| \frac{1}{\sqrt{T_n}}\sumi \left(\p_\nu^l\rho(\ve_i;\nu) - \E[\p_\nu^l \rho(\ve_1;\nu)]\right) \right|^K
\right] < \infty
\label{hm:tpe_lem1-6}
\end{equation}
for any $\nu\in\overline{\Theta_\nu}$ and $l\in\mbbzp$ through the the Sobolev-inequality argument as before.
Hence it suffices to show that there exists a constant $c'_2\in(0,1/2]$ such that
\begin{equation}
\sup_n \E\left[
\sup_\nu \left|T_n^{c'_{2}}\left(
\frac{1}{T_n}\sumi \left( \rho(\hep_i;\nu) - \rho(\ve_i;\nu) \right)\right)\right|^K
\right] < \infty.
\label{hm:tpe_lem1-1}
\end{equation}

To manage the term ``$\rho(\hep_i;\nu) - \rho(\ve_i;\nu)$'', we will separately consider it on different four events.
Let
\begin{align}
H_{1,i} := \left\{|\hep_i - \ve_i| \le \frac12 |\ve_i|\right\}, \qquad i=1,\dots,[T_n].
\nonumber
\end{align}
Writing $\hat{u}_{\mu,n}=\sqrt{N_n}(\mes-\mu_0)$ and $\hat{u}_{\sig,n}=\sqrt{N_n}(\ses-\sig_0)$, we have
\begin{align}
\hep_i - \ve_i = \frac{-1}{\ses \sqrt{N_n}}\left( \ve_i \hat{u}_{\sig,n} + (X_i-X_{i-1})\cdot \hat{u}_{\mu,n} \right).
\label{hm:hep-ep.estimate}
\end{align}
Let $(b_n)_n$ be a positive sequence tending to infinity. Then,
\begin{align}
H_{1,i}^c
&\subset \left\{
|\hat{u}_{\sig,n}||\ve_i| + |X_i - X_{i-1}||\hat{u}_{\mu,n}| \ge C \sqrt{N_n} |\ve_i|
\right\} 
\nn\\
&\subset \left\{ (|X_i - X_{i-1}| + |\ve_i|)(|\hat{u}_{\mu,n}| + |\hat{u}_{\sig,n}|) \ge C \sqrt{N_n} |\ve_i| \right\} 
\nn\\
&\subset \{|\ve_i| \le 1\} \cup 
\left\{|\ve_i| \ge 1,~
|\hat{u}_{\mu,n}| + |\hat{u}_{\sig,n}| \ge C\sqrt{N_n}\, \frac{|\ve_i|}{|\ve_i|+b_n}\right\} 
\nn\\
&{}\qquad\cup \left\{ |X_i - X_{i-1}| \ge b_n  \right\} 
\nn\\
&\subset \{|\ve_i| \le 1\} \cup 
\left\{|\hat{u}_{\mu,n}| + |\hat{u}_{\sig,n}| \ge C \frac{\sqrt{N_n}}{b_n}\right\} \cup 
\left\{ |X_i - X_{i-1}| \ge b_n  \right\} 
\nn\\
&=: H_{2,i} \cup H_{3,n} \cup H_{4,i}.
\nonumber
\end{align}

First, for $H_{1,i}$ and $H_{2,i}$, we take a closer look at the right-hand side of the expression
\begin{equation}
\rho(\hep_i;\nu) - \rho(\ve_i;\nu) = \left(\int_0^1 \p_\ep \rho(\ve_i + s(\hep_i - \ve_i);\nu)ds \right) (\hep_i -\ve_i).
\label{hm:tpe_lem1-2}
\end{equation}
On $H_{1,i}$, we have
\begin{equation}
\inf_{s\in[0,1]} \left| \ve_i + s(\hep_i-\ve_i)\right| \ge 
\frac12 |\ve_i|.
\nonumber
\end{equation}
Hence by \eqref{hm:rho.prop-1} and \eqref{hm:hep-ep.estimate},
\begin{align}
\left|\rho(\hep_i;\nu) - \rho(\ve_i;\nu)\right|I_{H_{1,i}}
&\lesssim 
\frac{1}{\sqrt{N_n}}\left( |\hat{u}_{\sig,n}| + \frac{|X_i-X_{i-1}|}{1+|\ve_i|}|\hat{u}_{\mu,n}| \right),
\label{hm:t_tpe_lem1-1}\\
\left|\rho(\hep_i;\nu) - \rho(\ve_i;\nu)\right|I_{H_{2,i}}
&\lesssim 
\frac{1}{\sqrt{N_n}}\left( |\hat{u}_{\sig,n}| + |X_i - X_{i-1}||\hat{u}_{\mu,n}|  \right).
\end{align}
The last two displays together with the tail-probability estimate \eqref{hm:CQMLE.tpe} and \eqref{T/N->0} imply that
\begin{align}
\E\left[
\sup_\nu \left|\sqrt{T_n}\!\left(\frac{1}{T_n}\sumi \left( \rho(\hep_i;\nu) - \rho(\ve_i;\nu) \right)I_{H_{1,i}\cup H_{2,i}}\right)\!\right|^K
\right] 
\lesssim \left(\frac{T_n}{N_n}\right)^{\frac{K}{2}} \!\!\! \to 0.
\label{hm:t_tpe_lem1-4}
\end{align}

Second, we consider $H_{3,n}$ and $H_{4,i}$. To this end, we do not use the expression \eqref{hm:tpe_lem1-2} but directly estimate the target expectation.
Recalling \eqref{hm:CQMLE.tpe}, for any $K_1,K_2>0$ we obtain
\begin{align}
\max_{1\le i\le [T_n]}\pr\left[H_{3,n}\cup H_{4,i}\right]
&\lesssim 
\left(\frac{b_n}{\sqrt{N_n}}\right)^{K_1}
\E\left[|\hat{u}_{\mu,n}|^{K_1} + |\hat{u}_{\sig,n}|^{K_1}\right]
\nn\\
&{}\qquad + b_n^{-K_2}\E\left[|X_i - X_{i-1}|^{K_2} \right]
\nn\\
&\lesssim \left(\frac{b_n}{\sqrt{N_n}}\right)^{K_1} + b_n^{-K_2}.
\nonumber
\end{align}
Now we set $b_n=N_n^{1/4}$. 
Given any $K>0$, under \eqref{T/N->0} we can take sufficiently large $K_1,K_2>0$ so that
\begin{equation}
\max_{1\le i\le [T_n]}\pr\left[H_{3,n}\cup H_{4,i}\right] = o\left(T_n^{-K}\right).
\label{hm:t_tpe_lem1-2}
\end{equation}
\hmrrev{
We will use the following (rough) estimate: there exists a constant $C'>0$ only depending on $K$ and $\nu_0$ for which $\{\log(1+\ep^2)\}^{2K} \le C'(1+|\ep|^{\nu_0/2})$ ($\ep\in\mbbr$).
}
With this observation together with \eqref{hm:rho.prop-1} and \eqref{hm:hep-ep.estimate}, we have
\begin{align}
& \frac{1}{T_n}\sumi\E\left[  \sup_\nu \left|\rho(\hep_i;\nu) - \rho(\ve_i;\nu)\right|^{2K} \right]^{1/2}
\nn\\
&\lesssim 
1 + \frac{1}{T_n}\sumi \left( \E\left[  \{\log(1+\hep_i^2)\}^{2K} + \{\log(1+\ve_i^2)\}^{2K} \right] \right)^{1/2}
\nn\\
&\lesssim 
1 + \frac{1}{T_n}\sumi \left( \E\left[  |\hep_i|^{\nu_0 /2} + |\ve_i|^{\nu_0 /2} \right] \right)^{1/2}
\nn\\
&\lesssim 
1 + \frac{1}{T_n}\sumi \left( \E\left[ |\ve_i|^{\nu_0 /2} 
+\left(
\frac{1}{\sqrt{N_n}}\left(|\ve_i||\hat{u}_{\sig,n}| + |X_{i}-X_{i-1}||\hat{u}_{\mu,n}|
\right)
\right)^{\nu_0 /2}
\right] \right)^{1/2}
\nn\\
&\lesssim 1.
\label{hm:t_tpe_lem1-3}
\end{align}
Combining \eqref{hm:t_tpe_lem1-2} and \eqref{hm:t_tpe_lem1-3}, \yurev{Jensen's inequality and Cauchy-Schwarz inequality yield}
\begin{align}
& \E\left[
\sup_\nu \left|\sqrt{T_n}\left(\frac{1}{T_n}\sumi \left( \rho(\hep_i;\nu) - \rho(\ve_i;\nu) \right)I_{H_{3,n}\cup H_{4,i}}\right)\right|^K 
\right]
\nn\\
&\lesssim 
T_n^{K/2}
\E\left[ \frac{1}{T_n}\sumi \sup_\nu \left|\rho(\hep_i;\nu) - \rho(\ve_i;\nu)\right|^K I_{H_{3,n}\cup H_{4,i}} \right]
\nn\\
&\yurev{\lesssim 
T_n^{K/2}
 \frac{1}{T_n}\sumi \E\left[\sup_\nu \left|\rho(\hep_i;\nu) - \rho(\ve_i;\nu)\right|^{2K}\right]^{1/2} \pr\left[H_{3,n}\cup H_{4,i}\right]^{1/2}}
\nn\\
&\lesssim 
T_n^{K/2}
\max_{1\le i\le [T_n]}\pr\left[H_{3,n}\cup H_{4,i}\right]^{1/2}
\frac{1}{T_n}\sumi\E\left[  \sup_\nu \left|\rho(\hep_i;\nu) - \rho(\ve_i;\nu)\right|^{2K} \right]^{1/2}
\nn\\
&\lesssim o(1).
\label{hm:t_tpe_lem1-5}
\end{align}
The claim \eqref{hm:tpe_lem1-1} with $c_2'=1/2$ follows from \eqref{hm:t_tpe_lem1-4} and \eqref{hm:t_tpe_lem1-5}.
\end{proof}

As in the proof of \eqref{hm:CQMLE.tpe} in Theorem \ref{hm:thm_CQLA}, the claim \eqref{hm:tQMLE.tpe} follows on showing the next lemma.

\begin{lem}
\label{hm:tpe_lem2}
~
There exist a constant $c'_1\in(0,1/2)$ such that for every $K>0$,
\begin{equation}
\sup_n 
\left(T_n^{c'_1}\left|\Gam_{\nu,n} - \Gam_{\nu,0} \right|
+
\sup_{\nu}\left| \frac{1}{T_n} \p_{\nu}^{3}\mbbh_{2,n}(\nu)\right|
\right)
+
\sup_n \E\left[|\D_{\nu,n}|^K\right] 
< \infty,
\nonumber
\end{equation}
where $\Gam_{\nu,n} := - T_n^{-1}\p_\nu^2 \mbbh_{2,n}(\nu_0)$.
\end{lem}

\begin{proof}
By \eqref{hm:ep.hat-ep}, the finiteness of the first (non-random) term is trivial for $c_1'\in(0,1]$.
Turning to the second term, we have
\begin{equation}
\D_{\nu,n}
= \frac{1}{\sqrt{T_n}}\sumi \p_\nu\rho(\ve_i;\nu_0)
+\frac{1}{\sqrt{T_n}}\sumi \left( \p_\nu\rho(\hep_i;\nu_0) - \p_\nu\rho(\ve_i;\nu_0) \right)
\label{hm:tpe_lem2-1}
\end{equation}
Let $K>0$.
Since $\E[\p_\nu\rho(\ve_i;\nu_0)]=0$, the first term is obviously $L^K(\pr)$-bounded.
Exactly in the same way as in the proof of \eqref{hm:tpe_lem1-1} with $c_2'=1/2$, the second term converges in $L^K$ to $0$.
\end{proof}


\subsubsection{Proof of (\ref{hm:tQMLE.an}): Asymptotic normality}
\label{hm:sec_t.an}

To prove the asymptotic normality \eqref{hm:tQMLE.an}, we introduce the concave random function
\begin{equation}
\Lam_n(u) := \mbbh_{2,n}\left(\nu_0 + \frac{u}{\sqrt{T_n}}\right) - \mbbh_{2,n}(\nu_0)
\label{hm:t.Lam_def}
\end{equation}
defined for $u\in\{v\in\mbbr:\,\nu_0 + v\,T_n^{-1/2} \in \Theta_\nu\}$; obviously, $\hat{u}_{\nu,n} \in \argmax \Lam_n$.
By means of \cite[Basic Lemma]{HjoPol93}, we can conclude that
\begin{equation}
\hat{u}_{\nu,n} = \Gam_{\nu,0}^{-1}\D_{\nu,n} + o_p(1) \cil N\left(0,\Gam_{\nu,0}^{-1}\right)
\nonumber
\end{equation}
by showing the locally asymptotically quadratic structure: for each $u\in\mbbr$,
\begin{equation}
\Lam_n(u) = \D_{\nu,n} u - \frac12 \Gam_{\nu,0}u^2 + o_p(1),
\label{hm:nu.laq}
\end{equation}
where the random variable $\D_{\nu,n} \cil N(0,\Gam_{\nu,0})$.
By Lemma \ref{hm:tpe_lem2} we are left to show the asymptotic normality of $\D_{\nu,n}$.
But it is trivial since $\ep \mapsto \p_\nu\rho(\ep;\nu)$ is smooth uniformly in a neighborhood of $\nu_0$ so that $\E[\{\p_\nu\rho(\ve_1,\nu_0)\}^2] = -\E[\p^2_\nu\rho(\ve_1,\nu_0)]=\Gam_{\nu,0}$.

\subsection{Proof of Theorem \ref{hm:thm_CQMLE.AN}}

It suffices to show \eqref{hm:CQMLE.an} and \eqref{hm:tQMLE.an} individually, \yurev{the same arguments} as in Section \ref{hm:sec_jan} are valid to deduce the asymptotic orthogonality of $\hat{u}_{a,n}$ and $\hat{u}_{\nu,n}$.

For the consistency of $\hat{a}_n$, it suffices to verify $\sup_a| \mbby_{1,n}(a) - \mbby_1(a)| \cip 0$.
In Section \ref{hm:sec_proof_C.lem1}, we considered the expression
\begin{equation}
\mbby_{1,n}(a) - \mbby_{1}(a) = 
\mbby^{\ast\ast}_{1,n}(a) - \mbby_{1}(a) 
+ \frac{1}{\sqrt{N_n}} \sumjj \frac{1}{\sqrt{N_n}}\widetilde{\zeta}_j(a) + \del_{11,n}(a) + \del_{12,n}(a).
\nonumber
\end{equation}
Assumption \ref{hm:A_X_AN} implies that $N_n^{-1}\sum_{j=1}^{N_n} |h^{-1}( \D_j X - h X'_{t_{j-1}})|^2 = o_p(1)$, hence $\sup_a |\del_{11,n}(a)| = o_p(1)$ from the expression \eqref{hm:def_del11a}.
Recall \eqref{hm:def_del.Y.ast} and \eqref{hm:cqmle_lem1-p1}.
By the definition \eqref{hm:zeta_def} we have 
\begin{equation}
\sup_a |\zeta_j(a)| \lesssim 1+ \log(1+\ep_j^2) + \log(1+|X'_{t_{j-1}}|^2) \lesssim 1+ \log(1+\ep_j^2),
\nonumber
\end{equation}
so that Lemma \ref{hm:lem_loc.lim.thm} gives $\sup_a |\del_{12,n}(a)| = o_p(1)$.
Since we also have $\sup_a |\p_a \zeta_j(a)| \lesssim 1+|X'_{t_{j-1}}| \lesssim 1$, the Burkholder and Sobolev inequalities yield $\sup_a |\sumjj N_n^{-1/2}\widetilde{\zeta}_j(a)|=O_p(1)$.
Finally, \yurev{from \eqref{hm:cqmle_lem2-p1}}, we have $\mbby^{\ast\ast}_{1,n}(a) - \mbby_{1}(a) \cip 0$ for each $a$ and moreover,
\begin{equation}
\sup_a |\p_a\mbby^{\ast\ast}_{1,n}(a)| + \sup_a |\p_a \mbby_{1}(a)| 
\lesssim \frac{1}{N_n^{-1}}\sum_{j=1}^{N_n}\left|\frac1h \D_j X\right| + 1 = O_p(1).
\nonumber
\end{equation}
Hence $\sup_a |\mbby^{\ast\ast}_{1,n}(a) - \mbby_{1}(a)| \cip 0$, followed by $\sup_a| \mbby_{1,n}(a) - \mbby_1(a)| \cip 0$.
For the asymptotic normality, 
\hmrev{
we can follow the same line as in the proof of Lemma \ref{hm:cqmle_lem3} (Section \ref{hm:sec_aAN_lem_proof}) under Assumption \ref{hm:A_X_AN} except for the following point:
concerned with the second term in \eqref{hm:Da-p1} about the first-stage quasi-score $\D_{a,n}$, we here do not require the moment boundedness (the last condition in Assumption \ref{hm:A_X}.2) but just its negligibility.
Recall the notation $R_{n,j}= h^{-1}(\D_j X  - h X'_{t_{j-1}})$ and $g_1(\ep)=\p_{\ep}\phi_1(\ep)/\phi_1(\ep)$.
By the compensation of $R_{n,j}\,g_1(\ep_j)$ and then applying the Burkholder inequality, it is straightforward to verify the negligibility as follows:
for the martingale part $M'_n := \sumjj N_n^{-1/2}\{R_{n,j}\,g_1(\ep_j) - \E^{j-1}[R_{n,j}\,g_1(\ep_j)]\}$, we can estimate its quadratic characteristic as
\begin{align}
\la M'\ra_n \lesssim \frac{1}{N_n} \sumjj \E^{j-1}[|R_{n,j}|^2] \cip 0
\nonumber
\end{align}
by Assumption \ref{hm:A_X_AN}.3; and moreover, the compensation part is negligible under Assumption \ref{hm:A_X_AN}.5.
}
Thus, we conclude \eqref{hm:CQMLE.an}.

To deduce \eqref{hm:tQMLE.an} under $\sqrt{N_n}(\hat{a}_n - a_0)=O_p(1)$, we make use of what we have seen in Sections \ref{hm:sec_t.tpe} and \ref{hm:sec_t.an}.
The $u$-wise asymptotically quadratic structure \eqref{hm:nu.laq} of $\Lam_n(\cdot)$ holds since the estimate $\sup_n |\Gam_{\nu,n} - \Gam_{\nu,0} | + \sup_{n,\nu}| T_n^{-1}\p_{\nu}^{3}\mbbh_{2,n}(\nu)| < \infty$ is valid as before.
Note that
\begin{align}
\frac{1}{T_n}\sumi | X_i - X_{i-1}| &= \frac{1}{n T_n}\sumi \sum_{j\in A_i} \left| \frac1h \D_j X \right|
\nn\\
&\lesssim \left(\frac{1}{[n T_n]}\sum_{j=1}^{[n T_n]} \left|\frac1h \left( \D_j X - h X'_{t_{j-1}}\right)\right|^2 \right)^{1/2} + 1
=O_p(1).
\nonumber
\end{align}
By \eqref{hm:rho.prop-1} and \eqref{hm:hep-ep.estimate}, the second term in the right-hand side of \eqref{hm:tpe_lem2-1} can be bounded as follows:
\begin{align}
& \left|\frac{1}{\sqrt{T_n}}\sumi \left( \p_\nu\rho(\hep_i;\nu_0) - \p_\nu\rho(\ve_i;\nu_0) \right)\right|
\nn\\
&= \left|\frac{1}{T_n}\sumi \int_0^1 \p_\ep\p_\nu 
\rho\left(\ve_i + s(\hep_i-\ve_i);\nu_0\right)ds\,\sqrt{N_n}(\hep_i-\ve_i) \sqrt{\frac{T_n}{N_n}}\right|
\nn\\
&\lesssim 
\left(|\hat{u}_{\sig,n}|\, \frac{1}{T_n}\sumi |\ve_i| + |\hat{u}_{\mu,n}| \, \frac{1}{T_n}\sumi (1+|X_i - X_{i-1}|) \right)\sqrt{\frac{T_n}{N_n}} 
\nn\\
&= O_p\left(\sqrt{\frac{T_n}{N_n}}\right)=o_p(1).
\nonumber
\end{align}
This leads to \eqref{hm:tQMLE.an}.

\subsection{Proof of Theorem \ref{hm:thm_CQMLE.AN-SDE}}
\label{hm:proof_SDE}

First, we will verify Assumption \ref{hm:A_X_AN} (except for item 4 therein) under Assumption \ref{hm:assump_SDE}.
\hmrev{
First, we show items 1 and 2 partly related to the ergodic behavior of the Markov process $Y$.
}
The {\lm} $\nu(dz)$ of $J$ satisfies that $\nu(\{z:\,|z|<c)>0$ for each $c>0$ (see \eqref{hm:J.Ld}).
\hmrev{
From the proof of \cite[Proposition 5.4]{Mas13as} under Assumption 5.2(ii) therein, we see that the Local Doeblin (LD) condition holds for $Y$ (see \cite{Kul09} for details). The LD condition implies that, for any constant $\D>0$, every compact set is petite for the discrete-time Markov chain $(Y_{m\D})_{m\ge 0}$.
Moreover, the {\lp} $J$ is centered and the function $\psi(y)=y^2$ belongs to the domain of the (extended) generator $\mca$ of $Y$:
\begin{align}
\mca \psi(y) &:= \p\psi(y)\, \mu_0\cdot b(y) + \int \left(
\psi(y+\sig_0 z) - \psi(y) - \p\psi(y) \sig_0 z\right)g(z)dz
\nn\\
&
=2y \mu_0\cdot b(y) + \frac{\sig_0^2}{\nu_0-2}.
\nonumber
\end{align}
By the latter condition in Assumption \ref{hm:assump_SDE}, we can find constants $K,c>0$ such that $\mca \psi(\yurev{y}) \le -c$ for every $|y|\ge K$; as a matter of fact, it suffices that the rightmost side is negative uniformly in $|y|\ge K$, but then the condition involves the unknown $\sig_0$ and $\nu_0$ hence inconvenient.
Obviously, there exists a constant $d'=d'(K)>0$ for which $\mca \psi(\yurev{y}) \le d$ for every $|y|\le K$.
We get $\mca \psi(\yurev{y}) \le -c + (c+d') I(|y|\le K)$ for $y\in\mbbr$.
Hence, it follows from \cite[Theorem 2.1]{Mas07} that $Y$ is ergodic, which in particular means that there exists a unique invariant distribution $\pi_{0,Y}(dy)$ for which $Y$ satisfies the law of large numbers: for any continuously differentiable \yurev{$\pi_{0,Y}$-integrable} $g:\mbbr\to\mbbr$, we have \yurev{$T^{-1} \int_0^T g(Y_s)ds \to \int g(y)\pi_{0,Y} (dy)$} as $T\to\infty$. The last convergence holds for the bounded process $X'=b(Y)$ as well.
It is routine to show that $N_n^{-1}\sumjj g\circ b (Y_{t_{j-1}}) \cip \int g\circ b (y)\pi_{0,Y}(dy)$ for $g$ being bounded and smooth enough.
Thus, items 1 and 2 hold with $\pi_0(dx')=\pi_{0,Y}\circ b^{-1}(dx')$.
}

Turning to item 3, we will show $\lim_n N_n^{-1}\sum_{j=1}^{N_n}A_{nj}=0$, where
\begin{equation}
A_{nj} := \E\left[\left| \frac1h \left( \D_j X - h X'_{t_{j-1}}\right) \right|^2\right]
=\E\left[\left| \frac1h \int_{t_{j-1}}^{t_j} (b(Y_s) - b(Y_{t_{j-1}}))ds \right|^2\right].
\nonumber
\end{equation}
Fix $\ep>0$ and $H\in(0,1\wedge \nu_0)=(0,1)$. Pick a $\del>0$ for which $|b(x)-b(y)|<\sqrt{\ep}$ for every $x,y$ such that $|x-y|\le\del$.
By using the boundedness of $b$,
\begin{align}
A_{nj} &\le \frac1h \int_{t_{j-1}}^{t_j} \E\left[|b(Y_s) - b(Y_{t_{j-1}})|^2\right] ds
\nn\\
&\lesssim \frac1h \int_{t_{j-1}}^{t_j} \big(\pr[|Y_s-Y_{t_{j-1}}|>\del]
+\E\left[|b(Y_s) - b(Y_{t_{j-1}})|^2;\, |Y_s-Y_{t_{j-1}}|\le\del \right]\big)ds
\nn\\
&\le \frac1h \int_{t_{j-1}}^{t_j} \pr[|Y_s-Y_{t_{j-1}}|>\del]ds + \ep
\nn\\
&\lesssim \frac1h \int_{t_{j-1}}^{t_j} \left((s-t_{j-1})^H + \E\left[|J_{s-t_{j-1}}|^H\right]\right)ds + \ep
\nn\\
&\lesssim O(h^{H}) + \sup_{t\in[0,h]}\E\left[|J_{t}|^H\right] + \ep = O(h^{H}) + \ep.
\label{hm:SDE_p1}
\end{align}
In the last step, we applied \cite[Theorem 2(c)]{LusPag08} to conclude $\E[\sup_{t\in[0,h]}|J_{t}|^H] \lesssim h^{H}$.
Since $\ep>0$ was arbitrary, it follows that $\lim_n N_n^{-1}\sum_{j=1}^{N_n}A_{nj}=0$, hence item 3.

\hmrev{
Finally, we verify item 5.
Pick any constant $\del' \in (0,(1+\ep')/2]$ for $\ep'>0$ in \eqref{hm:B_order}.
By the essential boundedness of $b$, $\p b$, and $g_1$ and by the moment estimate
\begin{equation}
\E[|J_s-J_{t_{j-1}}|] \le \E\left[\sup_{t\in[0,h]}|J_{t}|\right] \lesssim h\log(1/h).
\nonumber
\end{equation}
due to \cite[Theorem 3]{LusPag08}, we have
\begin{align}
& \left|
\frac{1}{\sqrt{N_n}}\sumjj \E^{j-1}\left[ \frac1h \left(\D_j X  - h X'_{t_{j-1}}\right)
\, g_1\left(\frac{\D_j J}{h}\right)\right]
\right|
\nn\\
&\lesssim \frac{1}{N_n} \sumjj \sqrt{N_n} \left( h + \frac1h \int_j \E[|J_s-J_{t_{j-1}}|] ds \right)
\nn\\
&\lesssim 
\sqrt{N_n} \big( h + h \log(1/h) \big)
\lesssim h^{1-\del'} \sqrt{N_n} \times h^{\del'}\log(1/h) \lesssim h^{1-\del'} \sqrt{N_n} \times h^{\del'/2}.
\nonumber
\end{align}
Since $h^{1-\del'} \sqrt{N_n} \lesssim n^{\del'-(1+\ep')/2} \lesssim 1$, the last upper bound is $O_p(h^{\del'/2})=o_p(1)$, concluding item 5.
Thus, we have shown that Assumption \ref{hm:A_X_AN} follows from Assumption \ref{hm:assump_SDE}.
}

\medskip

We will complete the proof of Theorem \ref{hm:thm_CQMLE.AN-SDE} by showing that $\sqrt{N_n}(\tilde{a}_n - \hat{a}_n)=o_p(1)$ and $\sqrt{T_n}(\tilde{\nu}_n - \nes)=o_p(1)$.
Let $\del_n(a):=\tilde{\mbbh}_{1,n}(a) - \mbbh_{1,n}(a)$.
Tracing back and inspecting the proof of Theorem \ref{hm:thm_CQMLE.AN} reveal that it is sufficient for concluding $\sqrt{N_n}(\tilde{a}_n - \hat{a}_n)=o_p(1)$ to verify
\begin{align}
\left| \frac{1}{\sqrt{N_n}}\p_a \del_n(a_0)\right| + \max_{k\in\{0,2,3\}}\sup_a\left| \frac{1}{N_n}\p_a^k \del_n(a)\right| = o_p(1).
\label{hm:SDE.proof.1-1}
\end{align}
We can write $\del_n(a) = \sum_{j=1}^{N_n} \{\log\phi_1(\yurev{\ep'_j(a)}) - \log\phi_1(\yurev{\ep_j(a)})\} = \sum_{j=1}^{N_n} B_j(a) (h^{-1}\D_j X - b_{j-1})$ for some essentially bounded random functions $B_j(a)$ smooth in $a$.
By direct computations and Assumption \ref{hm:A_X_AN}.3,
\begin{equation}
\max_{k\in\{0,2,3\}} \sup_a\left| \frac{1}{N_n}\p_a^k \del_n(a) \right| \lesssim 
\frac{1}{N_n}\sum_{j=1}^{N_n} \left|h^{-1}\D_j X - b_{j-1}\right| = o_p(1)
\nonumber
\end{equation}
for each $k\in\{0,2,3\}$. In a similar way to \eqref{hm:SDE_p1},
\begin{align}
\left| \frac{1}{\sqrt{N_n}}\p_a \del_n(a_0) \right|
&\lesssim \frac{1}{N_n}\sum_{j=1}^{N_n} \frac1h \int_{t_{j-1}}^{t_j}\sqrt{N_n}|b(Y_s)-b(Y_{t_{j-1}})|ds
\nn\\
&\lesssim 
\frac{1}{N_n}\sum_{j=1}^{N_n} \frac1h \int_{t_{j-1}}^{t_j}\sqrt{N_n}|Y_s-Y_{t_{j-1}}|ds
\nn\\
&\lesssim 
\frac{1}{N_n}\sum_{j=1}^{N_n} \frac1h \int_{t_{j-1}}^{t_j}
\sqrt{N_n}\left( (s-t_{j-1}) + |J_s - J_{t_{j-1}}|\right)ds.
\label{hm:SDE_p2}
\end{align}
Therefore, by applying \cite[Theorem 3]{LusPag08} again,
\begin{align}
\E\left[ \left| \frac{1}{\sqrt{N_n}}\p_a \del_n(a_0) \right| \right]
&\lesssim h \sqrt{N_n} + \E\bigg[\sup_{t\in[0,h]}|J_t|\bigg] 
\lesssim O\left(h \log(1/h)\sqrt{N_n}\right).
\nonumber
\end{align}
The upper bound in the last display is $o(1)$ since $h\sqrt{N_n}\lesssim n^{-\ep'/2}$.

To show $\sqrt{T_n}(\tilde{\nu}_n - \nes)=o_p(1)$, we write $\tilde{\mbbh}_{2,n}(\nu) = \sumi \rho(\tep_i;\nu)$.
From the proof of (\ref{hm:tQMLE.an}) (Section \ref{hm:sec_t.an}),
\begin{align}
\tilde{\Lam}_n(u) &:= 
\tilde{\mbbh}_{2,n}\left(\nu_0 + \frac{u}{\sqrt{T_n}}\right) - \tilde{\mbbh}_{2,n}(\nu_0)
\nn\\
&= \frac{u}{\sqrt{T_n}} \p_\nu\tilde{\mbbh}_{2,n}(\nu_0) - \frac12 \Gam_{\nu,0}u^2 + o_p(1),
\nonumber
\end{align}
Proceeding as in \eqref{hm:SDE_p2}, we obtain (with recalling \eqref{hm:rho.prop-1})
\begin{align}
\left|\frac{1}{\sqrt{T_n}} \p_\nu\tilde{\mbbh}_{2,n}(\nu_0) - \D_{\nu,n}\right|
&\lesssim \frac{1}{T_n} \sumi \left(\big|\sqrt{T_n}(\tilde{\ep}_i - \hat{\ep}_i)\big|
+\big|\sqrt{T_n}(\hat{\ep}_i - \ep_i)\big|\right)
\nn\\
&\lesssim O_p\left(h \log(1/h)\sqrt{T_n}\right) = O_p\left(n^{-\ep'/2}(\log n)\frac{T_n}{N_n}\right) = o_p(1).
\nonumber
\end{align}
It follows that $\sqrt{T_n}(\tilde{\nu}_n - \nes)=o_p(1)$.

\bigskip

\noindent
\bmhead{Acknowledgments}
The authors are grateful to the two reviewers for their careful reading, which led to substantial improvements in the presentation and based on which we could fix some insufficient technical arguments.
This work was supported by JST CREST Grant Number JPMJCR2115, Japan, by JSPS KAKENHI Grant Number 22H01139, and by JSPS KAKENHI Grant Number 23K13023.

\section*{Declarations}
\bmhead{Conflict of interest}
On behalf of all authors, the corresponding author states that there is no conflict of interest.



\end{document}